\documentclass[reqno, 12pt, dvipsnames]{amsart}     

\usepackage[foot]{amsaddr}
\usepackage{siunitx}
\usepackage{amsmath}
\numberwithin{equation}{section}
\usepackage{amsfonts,amssymb}
\usepackage{hyperref}
\usepackage{bm}
\usepackage{tikz,tkz-euclide,pdftexcmds}
\usetikzlibrary{arrows,shapes.geometric}
\usetikzlibrary{positioning}
\usetikzlibrary{automata}
\usetikzlibrary{arrows,shapes,calc}
\usepackage{float}
\usepackage[ruled,vlined]{algorithm2e}
\usepackage{verbatim}
\usepackage{thmtools, thm-restate}
\usepackage[font=small]{caption}
\usepackage[labelformat=simple,labelfont={}]{subcaption}

\usepackage{xcolor}
\usepackage[T1]{fontenc}
\usepackage{setspace} 
 
\newcommand{\com}[1]{\textcolor{red}{\texttt{giulio:} #1}}
\newcommand{\igor}[1]{\textcolor{blue}{\texttt{igor:} #1}}

\newcommand{\V}{{V}} 
\newcommand{\E}{{E}} 
\newcommand{\G}{G} 
\newcommand{\T}{\mathcal{T}} 
\newcommand{\tree}{T} 
\newcommand{\spatree}{\tau} 
\newcommand{\subtree}{\zeta} 
\newcommand{\Ssub}{S} 
\newcommand{\PP}{\mathbb{P}} 
\newcommand{\twostage}{\rm ts} 
\newtheorem{theorem}{Theorem}
\newtheorem{proposition}[theorem]{Proposition}
\newtheorem{conjecture}[theorem]{Conjecture}
\newtheorem{lemma}[theorem]{Lemma}
\newtheorem{corollary}[theorem]{Corollary}

\bibdata{prob}
\bibstyle{alpha}

\title[A transient equivalence between Aldous-Broder and Wilson's algorithms]{A transient equivalence between Aldous-Broder and Wilson's algorithms and a two-stage framework for generating uniform spanning trees}

\author{Igor Nunes$^{1,2}$}
\author{Giulio Iacobelli$^3$}
\author{Daniel Ratton Figueiredo$^2$}

\address{\SMALL $^1$Systems Engineering and Computer Science Program (PESC),  Federal University of Rio de Janeiro (UFRJ), Rio de Janeiro, Brazil.}

\address{\SMALL $^2$Department of Computer Science, UC Irvine (UCI), Irvine, USA}

\address{\SMALL $^3$Mathematical Institute,  Federal University of Rio de Janeiro (UFRJ), Rio de Janeiro, Brazil.} 

\email{\rm \texttt{igord@uci.edu, giulio@im.ufrj.br, daniel@cos.ufrj.br}}

\makeatletter


\begin{document}
\makeatletter
\pgfdeclareradialshading[tikz@ball]{my ball}{\pgfqpoint{5bp}{10bp}}{%
 color(0bp)=(tikz@ball!10!white);
 color(8bp)=(tikz@ball!65!white);
 color(12bp)=(tikz@ball);
 color(25bp)=(tikz@ball!85!black);
 color(50bp)=(tikz@ball!10!black)}

\pgfdeclareradialshading[tikz@ball]{new ball}{\pgfqpoint{12bp}{12bp}}{%
 color(0cm)=(tikz@ball!40!white);
 color(0.5cm)=(tikz@ball!80!white);
 color(0.75cm)=(tikz@ball);
 color(0.9cm)=(tikz@ball!80!black);
 color(1cm)=(tikz@ball!60!black)} 

\tikzoption{my ball color}{\pgfutil@colorlet{tikz@ball}{#1}\def\tikz@shading{my ball}\tikz@addmode{\tikz@mode@shadetrue}}

\tikzoption{new ball color}{\pgfutil@colorlet{tikz@ball}{#1}\def\tikz@shading{new ball}\tikz@addmode{\tikz@mode@shadetrue}}

\tikzoption{ball shadow color}{\pgfutil@colorlet{tikz@ball}{#1}\def\tikz@shading{ball shadow}\tikz@addmode{\tikz@mode@shadetrue}}
\makeatother

\newcommand{\ball}[4][white]{%
    \begin{scope}[shift = {(#3,#4)}]
      \shade[my ball color=#2] (0.15,0.25) circle (0.45);
    \end{scope}
    \node[color=black] at (.15+#3,.25+#4) {\Large$\mathbf{#1}$};
}

\maketitle

\begin{abstract}

The \emph{Aldous-Broder} and \emph{Wilson} are two well-known algorithms to generate uniform spanning trees (USTs) based on random walks. This work studies their relationship while they construct random trees with the goal of reducing the total time required to build the spanning tree. Using the notion of \textit{branches} -- paths generated by the two algorithms on particular stopping times, we show that the trees built by the two algorithms when running on a complete graph are statistically equivalent on these stopping times. This leads to a hybrid algorithm that can generate uniform spanning trees of complete graphs faster than either of the two algorithms. An efficient two-stage framework is also proposed to explore this hybrid approach beyond complete graphs, showing its feasibility in various examples, including transitive graphs where it requires 25\% less time than \emph{Wilson} to generate a UST. 



\end{abstract}
\section{Introduction}
\label{sec:introduction}

While a graph in general has an exponential number of distinct spanning trees, a fundamental and intriguing question is \textit{how does a ``typical'' spanning tree look like?} The search for an answer brought forth one of the most extensively explored probabilistic objects in graph theory: the \textit{uniform spanning tree} (UST). The \textit{uniform spanning tree} of a simple connected graph $\G = (\V,\E)$, denoted by UST$(\G)$, is the uniform probability measure on the set of all spanning trees of $\G$, denoted by $\T_{\G}$. The study of this object dates back to \textit{Kirchhoff's theorem} \cite{kirchhoff1847ueber}, which provides the size of the set $\T_{\G}$ in terms of the eigenvalues of the Laplacian matrix of $\G$.

From that time on, it has been shown that UST has surprising connections to a rich set of subjects in discrete probability, statistical mechanics and theoretical computer science. Amongst other topics, some examples are: random walks, percolation, gaussian fields, dimers and sandpile models~\cite{aldous2013probability,burton1993local,forman1993determinants,benjamini2001special,kassel2015transfer,jarai2015approaching,kenyon2011spanning}. Uniform spanning trees are also used in generating expander graphs, which have a large number of applications~\cite{hoory2006expander}. Furthermore, UST has played a crucial role in breaking long-standing limits in approximation algorithms for the Travelling Salesman Problem, in both symmetric and asymmetric variations~\cite{gharan2011randomized,asadpour2017log}.

Taking into consideration this wide assortment of applications, it is not surprising that the quest for designing fast algorithms that generate uniform spanning trees has received significant attention over the years. As a consequence of Kirchhoff's work, the first endeavors to algorithmically generate uniform spanning trees were the determinant-based approaches. The first algorithms generated a uniform spanning tree in running time $O(mn^{3})$, where $n=|V|$ and $m=|E|$~\cite{guenoche1983random,kulkarni1990generating}. It was only in the late 1980s that the random walk-based algorithms emerged as a much more promising approach. Such algorithms began with the following notable result due to Aldous \cite{aldous1990random} and Broder \cite{broder1989generating} (both acknowledging Diaconis for previous discussions and insights):
\begin{theorem}[Aldous-Broder]
    \label{thm:aldous-broder}
    Consider a simple random walk $(X_{t})_{t>0}$ on $\G=(\V,\E)$ starting from an arbitrary vertex $r\in \V$. Let $\E_X$ be the set of edges  $\bigcup_{t\geq 0} \big\{\{X_{t}, X_{t+1}\}$ such that $X_{t+1}\neq X_{k}$, for all $k\leq t \big\}$. Then $\tree=(\V, \E_X)$ is a random spanning tree of $\G$ with law {\rm UST}$(\G)$.
\end{theorem}

In order to generate a spanning tree of a graph $\G$, the random walk can stop at the first time it has visited all the vertices of $\G$, as the tree will be completed exactly when the last node is visited. This improved the running time of generating a UST$(\G)$ to the cover time of $G$ which has an expected value of $O(n\log n)$ for most graphs and $O(mn)$ for the worst cases \cite{broder1989bounds}. 
 

In order to build a spanning tree, the random walk clearly needs to visit every vertex of the
graph at least once. For this reason, it might seem unlikely that a random walk-based algorithm would generate a uniform spanning tree faster than the cover time of the graph. Nevertheless, Wilson~\cite{wilson1996generating} showed it to be possible through the following algorithm based on loop-erased random walks~\cite{lawler1980self}:

\vspace{5pt}\hspace{-.8cm}\textbf{Wilson}$\bm{(\G = (\V,\E),\tree_o)}$\textbf{:}
\begin{itemize}
    \item[0)] $V_X \leftarrow V(\tree_o)$ and $E_X \leftarrow E(\tree_o)$;\footnote{Given a graph $G$, let $\V(G)$ and $\E(G)$ denote its vertex and edge set, respectively.}\\
              If $\tree_o = \varnothing$ then $V_X \leftarrow \{ r \}$ with $r \in \V$ chosen arbitrarily;
    \item While $V_X \neq V$ :
        \begin{itemize}
            \item Choose $v \in V \setminus V_X$ uniformly at random
            \item Start a random walk on $v$ and stop when it first visits a node in $V_X$
            \item Let $p_v$ be the corresponding loop erasure path of the walker\footnote{In loop-erased random walks, the loops are erased in chronological order as they are created.}.
            \item Let $V_X = V_X \cup V(p_v)$ and $E_X = E_X \cup E(p_v)$
        \end{itemize}
    \item Return $\tree=(\V_X,\E_X)$.
\end{itemize}
Note that $\tree_o$ is an initial condition (initial tree, not spanning) of $G$, but for now assumed to be empty: $\tree_o = \varnothing$ (it will be used in Section~\ref{sec:hybrid}). Wilson showed that the expected running time of this algorithm is proportional to the \textit{mean hitting time} of $\G$. The mean hitting time $\mathbb{E}_u(h_{v})$ is the expected number of steps a random walk takes to reach vertex $v$ starting from vertex $u$ and the \textit{mean hitting time} of a graph $G$ is the mean of such times, considering all pairs of vertices. The mean hitting time of a graph is always less than the expected cover time (although its worst-case expected performance is still $O(mn)$~\cite{brightwell1990maximum}). Moreover, the mean hitting time can be considerably smaller for some graphs. On complete graphs for example, the expected cover time is $\Theta(n \log n)$, whereas the mean hitting time is $\Theta(n)$ (see Section~\ref{sec:hybrid} for a more detailed discussion). However, \textit{Matthew's bound} \cite{matthews1988covering} establishes that for any graph, the expected cover time is never larger than $H_{n-1}$ times the mean hitting time ($H_{n}$ is the $n$-th harmonic number which is $\Theta(\log n)$).

While both \textit{Aldous-Broder} and \textit{Wilson} algorithms are based on random walks, their behavior is fundamentally different in general. Consider the last few vertices to be added to the spanning tree; in \textit{Aldous-Broder} the random walk may wander a long time before covering these vertices, while in \textit{Wilson} they will be rather quickly added to the tree. In sharp contrast, when considering the first few vertices to be added to the spanning tree, \textit{Aldous-Broder} will quickly add them but \textit{Wilson} may wonder a long time and make loops until it reaches the first anchor node. Intuitively, \textit{Aldous-Broder} is more efficient in the beginning, to generate the initial vertices of the tree, and Wilson is more efficient in the end, to complete the spanning tree. Can this intuition be formally explored to combine the two approaches into a more efficient algorithm? This is the main theme addressed in this paper. 

One of the main contributions of this work is a positive result on combining the two algorithms. In fact, considering a complete graph with self-loops and an adequate sequence of stopping times (to be properly defined in Sections~\ref{sec:aldous} and~\ref{sec:wilson}), the two algorithms are shown to be equivalent in their transient behavior. More formally: 
\begin{restatable*}[\textit{Aldous-Broder} and \textit{Wilson} transient equivalence]{theorem}{equivalence}
Let \textit{Aldous-Broder} and \textit{Wilson} algorithms be initialized with the same vertex $r$ and let the input graph $\G$ be a complete graph with self-loops. Then, there exist two sequences of stopping times $(\hat{\sigma}_{i})_{i\geq0}$ (for \textit{Wilson}) and $(\sigma^{\rm in}_{i})_{i\geq0}$ (for \textit{Aldous-Broder}) such that, for every $i\geq 0$
    \[\tree^{W}_{\hat{\sigma}_{i}} \buildrel d \over = \tree^{AB}_{\sigma^{\rm in}_{i}}\;,
\]
where $\tree^{W}_{\hat{\sigma}_{i}}$ and $\tree^{AB}_{\sigma^{\rm in}_{i}}$ are the trees constructed by \emph{Wilson} and \emph{Aldous-Broder} by time $\hat{\sigma}_{i}$ and $\sigma^{\rm in}_{i}$, respectively.
\label{thm:equivalence}
\end{restatable*}
This result leads to a hybrid algorithm that starts with \textit{Aldous-Broder} and then switches to \textit{Wilson} to finish generating a spanning tree of a complete graph. Theoretical analysis and numerical simulations show that this approach has a smaller running time than either of the two algorithms (by a constant factor when considering \textit{Wilson}). 

The next contribution is a two-phase framework to explore such hybrid approach in a more general setting, beyond complete graphs. The first phase generates a tree of $G$ according to a distribution which satisfies a ``convenient'' condition (properly defined in Section~\ref{sec:framework}). The second phase finishes by generating a uniform random spanning tree given the initial tree. We show that Wilson's algorithm does the job for second phase, captured by the following lemma:
\begin{restatable*}{lemma}{secondstep}
Let $\G=(\V,\E)$  be a connected graph and  $\subtree$ be a tree (not spanning) of $\G$. Wilson's algorithm with initial condition  $\subtree$ returns a  spanning tree of $\G$ which is uniformly distributed on the set $\{\spatree \in \T_{\G}: \spatree \supseteq   \subtree\}$.
\label{lem:second_step}
\end{restatable*}
Again, \emph{Wilson} with an initial condition is a natural concept which consists in passing the tree $\subtree$ as the argument for $\tree_o$ (see Wilson's algorithm above). While generating an initial tree $\subtree$ that satisfies the condition for the first phase is not trivial, several examples where this is feasible and efficient are provided. For example, when considering any edge-transitive graph $G$, a random edge of $G$ satisfies the condition for the first phase (see Example~3 in Section~\ref{sec:framework}). Thus, the two-phase framework can be applied to generate UST in the hypercube with an average running time is asymptotically 25\% smaller than the average running time of \textit{Wilson} (shown numerically and through a conjecture, see Section~\ref{sec:hypercubes}). Intuitively, a larger initial tree will require less time for \textit{Wilson} finish generating the spanning tree in the second phase. However, even when the initial tree is just an edge the proposed framework delivers a significant improvement in the running time. Thus, an open challenge for broadening the use of the framework is generating adequate initial trees in other graphs, even if arbitrarily small. 


Of side interest, the framework can be adapted to generate random trees with probability proportional to the number of trees that are isomorphic to the initial tree provided as input (see Appendix~\ref{sec:linearly_biased}). More specifically, given a tree $\subtree$ with $k$ vertices, a random tree $\spatree$ with $n \geq k$ vertices can be generated in linear time with probability proportional to the number of sub-trees of $\spatree$ that are isomorphic to $\subtree$. This finds application in the area of generating random graphs (trees) that have a probability law for pre-defined subgraphs (sub-trees)~\cite{li2005sampling,deo1997computation,liu2005maximizing}. 

\subsection{Recent Advancements}

Wilson's algorithm opened the doors to algorithms that can generate UST more efficiently than the cover time of a random walk. Since then, a sequence of improved algorithms have emerged in the literature. Kelner and M{\k{a}}dry~\cite{kelner2009faster} proposed a method based on the Aldous-Broder algorithm for generating approximately uniform spanning trees by trying to skip unnecessary steps, i.e., steps over already-covered regions of the graph (named the \textit{shortcutting} strategy). The method returns a tree $\spatree$ with probability $p(\spatree)$, where $(1-\delta)/|{\mathcal T}_{G}| \leq p(\spatree) \leq (1+\delta)/|{\mathcal T}_{G}|$ for some $\delta > 0$, in $\tilde{O}(m\sqrt{n}\log 1/\delta)$ time ($\tilde{O}$ refers to the \textit{soft-O} notation, which is simply \textit{big-O} ignoring logarithmic factors~\cite{cormen2009introduction}). Later on, the strategy was proved to work for generating uniform (not approximately) spanning trees too~\cite{mkadry2011graphs}, resulting in a procedure with $\tilde{O}(m\sqrt{n})$ expected running time in general graphs.

Based on this last result and exploring the effective resistance metric~\cite{tetali1991random} and its relationship to random walks, M{\k{a}}dry \textit{et al.}~\cite{madry2015fast} proposed an approach that improved the performance to $\tilde{O}(m^{4/3})$ in general (and even better for sufficiently sparse graphs). Finally, based on these last results and a number of new facts about electrical flows, Schild~\cite{schild2018almost} introduced another shortcutting strategy that resulted in a method that generates an UST with an expected running time that is almost-linear on the number of edges, namely $O(m^{1+o(1)})$. It is worth mentioning that the effective resistance metric was also the basis for another recent and promising algorithm by Durfee \textit{et al.} \cite{durfee2017sampling}, which eschews from both determinant and random walk-based approaches and uses Gaussian elimination to generate USTs with high probability in $O(n
^{5/3}m^{1/3})$.

It is important to point out that while this work also addresses the efficient generation of USTs, our goal is not to compete directly with the state-of-the-art approaches. Instead, we provide a novel framework that is intended to give a new perspective into strategies for efficiently generating USTs, which can perhaps be combined with existing approaches (e.g., shortcutting) to yield even further improvements.

\section{Aldous-Broder branch distribution}
\label{sec:aldous}

This section starts by providing an interpretation for the evolution of the \emph{Aldous-Broder} algorithm that is instrumental in establishing its connection with the \emph{Wilson} algorithm. In particular, an  \textit{urn model} is defined and used to generate random spanning trees. This process is shown to be equivalent to \emph{Aldous-Broder} running on the complete graph (with self loops), including their transient behavior.

Lets start by precisely defining the \emph{Aldous-Broder} algorithm, and introducing some important notation:



%
\vspace{5pt}\hspace{-.8cm}\textbf{Aldous-Broder}$\bm{(\G = (\V,\E))}$\textbf{:}
\begin{itemize}
    \item[0)] Set ${\E_X} \leftarrow \varnothing$ and $X_0 \leftarrow r$, with $r \in \V$ chosen arbitrarily;
    \item While $|\E_X| \neq |\V| -1$:
        \begin{itemize}
            \item Run a simple random walk $(X_{t})_{t\geq 0}$ on $\G$, starting at $X_{0}$ and 
     for each edge $e = \{X_{t},X_{t+1}\}$ such that $X_{t+1} \neq X_{k}$ for all $k \leq t$,  set ${\E_X} \leftarrow {\E_X} \cup \{e\}$;
        \end{itemize}
    \item Return $\tree=(\V,\E_X)$.
\end{itemize}
Note that $r$ is a special vertex, chosen arbitrarily, and the starting point of the random walk. Moreover, the path taken by the walker $(X_{t})_{t\geq 0}$ can be divided into alternating sequences of vertices: vertices that the walker has already visited and vertices being visited for the first time. The following double sequence of stopping times capture the start time of each kind of excursion. Let $V_{t}$ denote the set of vertices of $\G$ visited by the random walk up to time $t$, i.e.,  $V_{t} = \bigcup_{k=0}^{t}\big\{X_{k}\big\}$, and $\eta$ the cover time of $\G$. Let $\sigma^{\rm in}_0\equiv 0$, and for $i>0$, 
\begin{align*}
\sigma^{\rm out}_{i}&:= \inf\{t>\sigma^{\rm in}_{i-1}: X_t \notin V_{t-1}\}\wedge \eta\;,
\\
\sigma^{\rm in}_{i}&:= \inf\{t>\sigma^{\rm out}_{i}: X_t \in V_{t-1}\}\wedge \eta \;,
\end{align*}
where $\sigma^{\rm in}$ and $\sigma^{\rm out}$ capture the start time of an excursion through vertices already visited and vertices being visited for the first time, respectively. Since $\G$ is finite and connected, $\eta$ is finite almost surely, and thus these stopping times are all finite. 


These stopping times can be used to define a sequence of \emph{branches} constructed by the random walk $(X_t)_{t\geq 0}$, which represents the excursions through vertices being visited for the first time. Let $B_i$ denote the $i$-th  branch, $i\geq 1$, and be defined by the sequence $\{X_{\sigma^{\rm out}_{i}}, \ldots, X_{(\sigma^{\rm in}_{i}-1)}\}$. A branch is a path graph that starts and ends on the first and last node of the sequence $B_i$, respectively. Note that the edge $\{X_{\sigma^{\rm out}_{i}-1}, X_{\sigma^{\rm out}_{i}}\}$ does not belong to any branch, but is the edge that connects branch $B_i$ to the set of nodes already visited. 

%

Figure~\ref{fig:branches} illustrates these definitions with an example of the evolution of the  \emph{Aldous-Broder} algorithm in a $5\times 5$ grid graph. The start vertex $r$ corresponds to vertex 16 and the construction of the first two branches is shown. The current position of the walker is represented by a square and the colored vertices are those already visited by the walker, where each color indicates a different branch. Bold edges represent the set of first-entrance edges used to construct the tree according to Theorem~\ref{thm:aldous-broder}. Dashed bold edges indicate the edges that  do not belong to any branch. 

\input{imgs/branches}

Let $\tree^{AB}_{\sigma^{\rm in}_i}$ be the random tree of $\G$ (not necessarily a spanning tree) constructed by the \emph{Aldous-Broder} algorithm when considering the first-entrance edges up to time $\sigma^{\rm in}_i$, i.e., $\bigcup_{t<\sigma^{\rm in}_i}\left\{\{X_t, X_{t+1}\} : X_{t+1}\neq X_k\, \forall k\leq t\right\}$. Note that $\tree^{AB}_{\sigma^{\rm in}_i}$ corresponds to the first $i$ branches along with their corresponding special edge that connects the branch to the existing tree. For example, the trees highlighted in Figures~\ref{fig:TAB1} and~\ref{fig:TAB2} are, respectively, $\tree^{AB}_{\sigma^{\rm in}_1}$ and $\tree^{AB}_{\sigma^{\rm in}_2}$. Figure~\ref{fig:branches} shows an example of all branches constructed by the \emph{Aldous-Broder} algorithm when running on a complete graph with a hundred vertices (each color is a branch). By the \textit{Aldous-Broder} theorem, this is a uniform spanning tree of $\G$. 

\begin{figure}
\centering
\makebox[\linewidth][c]{\includegraphics[width=.77\textwidth]{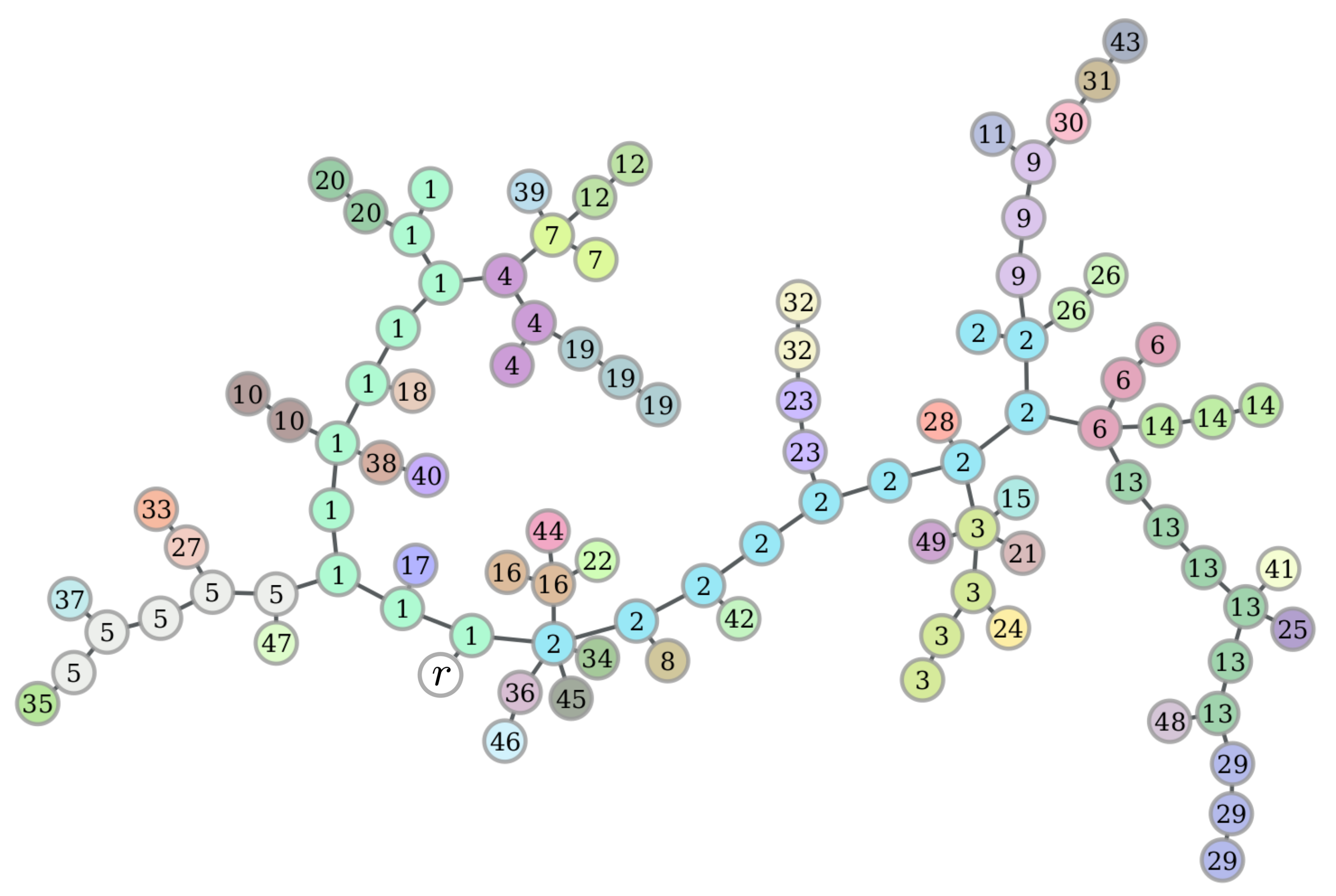}}
\caption{Branches of a uniform spanning tree of the complete graph with a hundred vertices generated by the \emph{Aldous-Broder} algorithm. Each color corresponds to a branch and vertices are numbered according to their branch number, with the exception of the initial vertex $r$ which does not belong to any branch.}
\end{figure}

\subsection{The urn model}
\label{sec:urn}

Consider an urn $U=\{b_1, \ldots, b_{n}\}$ containing $n$ uncolored but labelled balls, and let $\{c_{0}, \dots, c_{n-1}\}$ be a set of $n$ distinct colors. Consider the following color-assignment process, where $A$ denotes the set of uncolored balls at any point in time during the process: 
%
%
\begin{itemize}
    \setlength\itemsep{-1pt}
    \item[0)] Let $r \in \{1,\ldots,n\}$ be chosen arbitrarily; and paint $b_r$ with $c_0$; \\
              Set $A = U\setminus \{b_r\}$; \\
              Set $i \leftarrow 1$, let $b$ denote an uniformly chosen ball from $A$, paint $b$ with color $c_i$, return it to the urn;
    \item[1)] Let $b$ denote an uniformly chosen ball from $U$:
        \begin{itemize}
            \setlength\itemsep{-1pt}
            \item if $b\in A$ (the ball is uncolored), paint it with color $c_i$, return it to the urn;
            \item if $b\in U\setminus A$ (ball is already colored), return it to the urn, set $i \leftarrow i+1$, paint an uniformly chosen ball from $A$ with color $c_i$;
        \end{itemize}
    \item[2)] If $A \neq \emptyset$, repeat item $1)$, otherwise stop.    
\end{itemize}

Note that the color-assignment executes step $1)$ exactly $n-2$ times, since exactly one uncolored ball is colored each time a ball is drawn from the urn. Let $C_i$ denote the sequence of balls drawn that were painted with color $c_{i}$, for $i=0,\ldots,n-1$, where $C_i = \emptyset$ if color $c_i$ is not used. 
Figure~\ref{fig:urn} shows an example of the color-assignment process with $n=10$ balls, where $c_1 = $ \fcolorbox{black}{teal}{\rule{0pt}{3pt}\rule{3pt}{0pt}}, $c_2 = $ \fcolorbox{black}{Purple}{\rule{0pt}{3pt}\rule{3pt}{0pt}}, $c_3 = $ \fcolorbox{black}{Dandelion}{\rule{0pt}{3pt}\rule{3pt}{0pt}} and $c_4 = $ \fcolorbox{black}{RedOrange}{\rule{0pt}{3pt}\rule{3pt}{0pt}}:
\vspace{.8cm}

\begin{figure}[H]
\centering
\scalebox{.6}{
\begin{tikzpicture}
    \draw[line width=1] (-0.5,0.5) -- (-0.5,-0.3) -- (9.8,-0.3) -- (9.8,0.5);
        \ball[1]{Dandelion}{7}{0}
        \ball[2]{teal}{2}{0}
        \ball[3]{teal}{3}{0}
        \ball[4]{teal}{2}{0}
        \ball[5]{DarkOrchid}{5}{0}
        \ball[6]{white}{0}{0}
        \ball[7]{Orange}{8}{0}
        \ball[8]{teal}{1}{0}
        \ball[9]{DarkOrchid}{6}{0}
        \ball[10]{Orange}{9}{0}
        \ball[11]{teal}{4}{0}
\end{tikzpicture}}
\caption{Example of the color-assignment process with $n=10$ balls: the initial ball was 6 and four other colors were used.}
\label{fig:urn}
\end{figure}
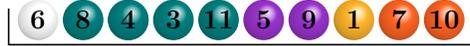

Let $|C_i|$ denote the number of balls colored with color $i$, where $|C_0| = 1$ by definition. The following lemma establishes the conditional distribution for this number given the total number of colored balls. 
\begin{lemma}\label{lem:urn-branch}
For every $i>0$, $k\in \{1, \ldots, n-1\}$, and $h \in \{1, \ldots, n-k\}$, 
\[\PP_U\left(|C_i|=h \left\vert\;  \sum_{j=0}^{i-1} |C_j|=k \right.\right)= \frac{(k+h)(n-k-1)!}{n^{h}(n-k-h)!} \;.\]
\end{lemma}
While its proof is straightforward, this lemma will be useful later when establishing the equivalence between \emph{Aldous-Broder} and \emph{Wilson} algorithms.

Let $C=(C_0, \ldots, C_{n-1})$ denote the random sequences of colored balls obtained from the color-assignment urn process. The following algorithm generates a random tree with $n$ vertices using the outcome of the color-assignment process for a urn with $n$ balls. Let $V = \{1,\ldots,n\}$ denote the set of vertices and let ball $b_j$ correspond to vertex $j$. 

\clearpage
\textbf{Urn-Tree}$\bm{(C)}$\textbf{:}
\begin{itemize}
    \item[0)] Set $\V_{T} \leftarrow \{r\}$, $\E_{T} \leftarrow \varnothing$ and $i \leftarrow 1$;
    \item While $C_i \neq \emptyset$: 
    \begin{itemize}
        \item Let $P$ denote a path graph according to the sequence $C_i$, where vertices correspond to balls and edges are placed between adjacent vertices in the sequence;
        \item Let $u$ denote an uniformly chosen vertex from $V_{T}$;
        \item Let $v$ denote the first ball of the sequence $C_i$;
        \item Let $\V_{T} \leftarrow \V_{T} \cup \V(P)$ and $\E_{T} \leftarrow \E_{T} \cup \E(P) \cup \{(u,v)\}$;
        \item $i \leftarrow i + 1$;
    \end{itemize}
    \item Return $\tree=(\V_{T}, \E_{T})$.
\end{itemize}

The \textbf{Urn-Tree} algorithm starts the tree construction process with the initial ball, which corresponds to the initial vertex. It then iterates through the used colors constructing a path graph according to the sequence of balls and attaching this path to the existing tree, with the the anchor vertex for the path chosen uniformly at random. The process finishes when all used colors have been considered. Figure~\ref{fig:urn-tree} shows an example of a tree generated by the \textit{Urn-Tree} algorithm when the input is the sequence of colored balls provided by Figure~\ref{fig:urn}.

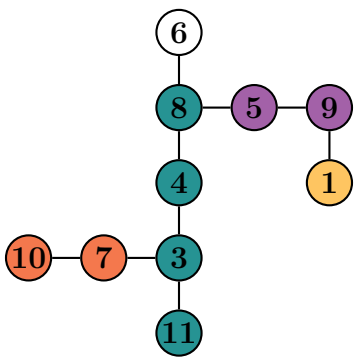
\begin{figure}
    \centering
    \scalebox{1}{\begin{tikzpicture}[auto,node distance=10mm, thick,main node/.style={circle,fill=blue!20,draw,minimum size=.6cm,inner sep=0pt]}]
        \node[main node,fill=white] (1) {$\mathbf{6}$};
        \node[main node,fill=teal!85] (2) [below of=1] {$\mathbf{8}$};
        \node[main node,fill=teal!85] (3) [below of=2] {$\mathbf{4}$};
        \node[main node,fill=teal!85] (4) [below of=3] {$\mathbf{3}$};
        \node[main node,fill=teal!85] (5) [below of=4] {$\mathbf{11}$};
        \draw[-] (1) edge (2);
        \draw[-] (2) edge (3);
        \draw[-] (3) edge (4);
        \draw[-] (4) edge (5);
        \node[main node,fill=Purple!85] (6) [right of=2] {$\mathbf{5}$};
        \node[main node,fill=Purple!85] (7) [right of=6] {$\mathbf{9}$};
        \draw[-] (6) edge (7);
        \draw[-] (6) edge (2);
        \node[main node,fill=Dandelion!85] (8) [below of=7] {$\mathbf{1}$};
        \draw[-] (8) edge (7);
        \node[main node,fill=RedOrange!85] (9) [left of=4] {$\mathbf{7}$};
        \node[main node,fill=RedOrange!85] (10) [left of=9] {$\mathbf{10}$};
        \draw[-] (9) edge (10);
        \draw[-] (9) edge (4);
\end{tikzpicture}}
    \caption{Example of a tree generate by the \textit{Urn-Tree} algorithm when using as input the sequence of colored balls shown in Figure~\ref{fig:urn}.}
    \label{fig:urn-tree}
\end{figure}



Consider the \emph{Aldous-Broder} algorithm on a complete graph with $n$ vertices and the \textit{Urn-Tree} on an urn with $n$ balls. While both algorithms build random trees with $n$ vertices, they correspond to significantly different random processes. An important difference is their running time. While the expected running time of \emph{Aldous-Broder} is $\Theta(n \log n)$ for a complete graph with $n$ vertices (see Section~\ref{sec:hybrid}), the color-assignment process and the \textit{Urn-Tree} algorithm has a worst case running time of $\Theta(n)$ for an urn with $n$ balls. 


%

Interestingly, the following theorem establishes not only that the \emph{Urn-tree} algorithm generates a uniform spanning tree of the complete graph with self-loops, but also that the \emph{Urn-tree} is equivalent to the transient behavior of the \textit{Aldous-Broder} algorithm running on a complete graph with self-loops in the stopping times that finish \emph{branches}. 

Let $\tree^{U}_{i}$ denote the random tree generated by \textit{Urn-tree} algorithm right after the attachment of the $i$-th path graph. Note that the number of vertices in $\tree^{U}_{i}$, henceforth denoted by $|\tree^{U}_{i}|$, is equal to $\sum_{j=0}^{i}|C_j|$. 

\begin{theorem}[Urn-Tree and Aldous-Broder transient equivalence]
The random tree $\tree^{AB}_{\sigma^{\rm in}_i}$ generated by \textit{Aldous-Broder} algorithm on a complete graph with self-loops with $n$ vertices and the random tree $\tree^{U}_{i}$ on an urn $U$ with $n$ balls, when initialized with the same vertex $r$, have the same distribution in the sense that $\tree^{U}_{i} \buildrel d \over = \tree^{AB}_{\sigma^{\rm in}_{i}}$ , for all $i=1, \ldots, n-1$ 


\label{thm:urn aldous equivalence}
\end{theorem}

\begin{proof}

Note that since both processes are initialized with the same special vertex, the first node of both random trees is the same, in particular $X_0 = C_0 = r$. Let $Z_i$ be the random variable denoting the number of edges in the $i$-th branch of a random walk $X$ on a complete graph with self-loops with $n$ vertices, as previously defined. Let $\PP_X$ denote the probability measure defined on the space of trajectories of $(X_t)_{t\geq 0}$, for $i>0$ and $k \in \{1,\ldots, n-1\}$,  
\begin{align*}
\PP_X\left(Z_i=h \; \left\vert \; |\tree_{\sigma^{\rm in}_{i-1}}|\right.=k \right) = \PP_U\bigg(|C_i|=h \;\Big\vert \;\sum_{j=0}^{i-1}|C_j|=k \bigg)\;,
\end{align*}
for every $h \in \{1,\ldots, n-k\}$. In essence, the random walk in Aldous-Broder can be interpreted as painting vertices as it traverses the graph until it hits an already painted vertex. Once it hits a painted vertex, the walker may wander around painted vertices and will eventually start another painting journey in a randomly chosen unpainted vertex. This process is identical to the urn process since the graph is complete and has self-loops. Thus, branch sizes and number of colored balls have the same corresponding conditional distribution. 

%

Recall that the \textit{Urn-Tree} algorithm attaches the path graph corresponding to the sequence $C_i$ to a uniformly chosen vertex in $\tree^U_{i-1}$. To finish the proof it is enough to show that this is also the case for the branches in Aldous-Broder. Recall that the edge $\{X_{\sigma^{\rm out}_{i}-1}, X_{\sigma^{\rm out}_{i}}\}$ connects the $i$-th branch constructed by Aldous-Broder to the existing tree. It suffices to show that the distribution of $X_{(\sigma^{\rm out}_{i-1}-1)}$ is uniform over the vertices of $\tree^{AB}_{\sigma^{\rm in}_{i-1}}$. Note that at time $\sigma^{\rm in}_{i-1}$, the $(i-1)$-th branch has just finished and since the graph is complete and has self-loops, $X_{\sigma^{\rm in}_{i-1}}$ is uniformly distributed over $\tree^{AB}_{\sigma^{\rm in}_{i-1}}$. 
At time $t=\sigma^{\rm in}_{i-1} +1$, either the $i$-th branch starts, i.e., $\sigma^{\rm out}_{i-1}=t$, which happens with probability $(n-|\tree^{AB}_{\sigma^{\rm in}_{i-1}}|)/n$,  or the random walk moves to a vertex of $\tree^{AB}_{\sigma^{\rm in}_{i-1}}$ with probability $|\tree^{AB}_{\sigma^{\rm in}_{i-1}}|/n$, choosing uniformly among the vertices in tree. This in particular implies that for every $k\geq 1$, the distribution of $X_{\sigma^{\rm out}_{i-1}-1}$ conditioned on $\sigma^{\rm out}_{i-1}=\sigma^{\rm in}_{i-1}+k$, is uniform  on $\tree^{AB}_{\sigma^{\rm in}_{i-1}}$. 
Thus, $X_{(\sigma^{\rm out}_{i-1}-1)}$  is uniform over the vertices of $\tree^{AB}_{\sigma^{\rm in}_{i-1}}$ and $\tree^{U}_{i} \buildrel d \over = \tree^{AB}_{\sigma^{\rm in}_{i}}$.

A final remark concerns the labels of the vertices in the trees $\tree^{U}_{i}$ and $\tree^{AB}_{\sigma^{\rm in}_{i}}$. The initial vertex chosen arbitrarily is the same for both trees.  For both processes, the labels used by a branch are drawn uniformly at random from the set of labels yet to be used (due to the complete graph and urn process). Thus, the probability that a given labelled tree is generated at the respective stopping times of each process is the same. 
\end{proof}

The transient equivalence between the \emph{Urn-Tree} and \emph{Aldous-Broder} algorithms established by the above theorem requires the complete graph to have self-loops. This ensures that every branch generated by \emph{Aldous-Broder} is connected to a uniformly chosen vertex among the already visited vertices. Without self-loops, $X_{\sigma^{\rm in}_{i}}$ must be different from $X_{(\sigma^{\rm in}_{i}-1)}$, an already visited vertex. However, the \emph{Urn-Tree} algorithm generates a uniform spanning tree of the complete graph either with or without self-loops (since self-loops are not part of any spanning tree).

\begin{corollary}
     The random tree $\tree^U_{n-1}$ produced by \textit{Urn-Tree} algorithm on an urn with $n$ balls is a uniform spanning tree of the complete graph with $n$ vertices.
\end{corollary}
\begin{proof}
    From Theorem~\ref{thm:urn aldous equivalence}, when $\G$ is a complete graph with self-loops, 
    $\tree^U_{i}\overset{d}{=} \tree^{AB}_{\sigma^{\rm in}_i}$, for every $i\leq n-1$. Then by Theorem~\ref{thm:aldous-broder} and observing that $\sigma^{\rm in}_{n-1}= \eta$, the result follows. If $\G$ is a complete graph without self-loops, the result also follows since UST($G$) is the same when either $G$ has self-loops or not.
\end{proof}

While the \emph{Urn-tree} algorithm is a linear time algorithm for generating USTs for the complete graph, it is clearly not the first. A simple procedure proposed by Aldous generates a UST for the complete graph using $n$ uniform choices on $\{1,\ldots,n\}$ and a random shuffle, thus also requiring time $\Theta(n)$ (see Algorithm 2 in \cite{aldous1990random}). However, the notion of branches in the \emph{Urn-tree} algorithm and its transient equivalence with \textit{Aldous-Broder} is instrumental in establishing the connection with Wilson's algorithm in the next section. 

\section{Wilson's branch distribution}
\label{sec:wilson}


Recall that Wilson's algorithm constructs a path (loop-erased walks) with vertices not yet in the tree prior to connecting the path to the existing tree. This suggests a very natural definition for \textit{branches} and stopping times: Let $\hat\sigma_0 \equiv 0$ and $\tree_{\hat \sigma_0}^{ W}$ the single initial vertex $r$ chosen arbitrarily, and recursively define $ \hat\sigma_{i}:= \inf\big\{t>\hat\sigma_{i-1}: X_{t}\in \tree^{W}_{\hat\sigma_{i-1}}\big\}$, where $\tree^{W}_{\hat\sigma_{i-1}}$ is the tree build by the \emph{Wilson} algorithm up to time $\hat \sigma_{i-1}$ for $i>0$. The $i$-th \emph{branch} corresponds to the loop erasure path $\hat p_i$ of the walk $(X_0^i, \ldots, X_{\hat \sigma_i - 1})$, where $X_0^i$ is a uniform vertex in $V\setminus V(\tree^{W}_{\hat\sigma_{i-1}})$. Note that the edge $\{X_{\hat \sigma_i - 1}, X_{\hat \sigma_i}\}$ does not belong to the branch but is the edge that connects the $i$-th branch to the existing tree $\tree^{W}_{\hat\sigma_{i-1}}$.

\subsection{Wilson's branch distribution  on  complete graphs}
\label{sec:wilson-complete}
While Wilson's algorithm produces branches for any graph, it is challenging to analyse their distribution in the general case. Thus, the complete graph is first considered, and the analysis unveils a strong relationship between \textit{Aldous-Broder} and \textit{Wilson} which will be formally established in Section~\ref{sec:equivalence}.

The first important property of the branches generated on complete graphs is that the labels of the vertices in any branch are simply a random subset of the labels. This is due to the fact that every vertex is structurally identical and also connected to every other vertex. So, the probability that a vertex appears in a specific branch, and in a specific order, does not depend on the vertex. This implies that the distribution of the branches is entirely defined by their lengths.

However, explicitly computing the branch length distribution in \textit{Wilson} algorithm is not trivial due to its loop erasing mechanism. Expressly, if $\hat p_i(t)$ denotes the $i$-th branch at time $t$ before completion, i.e.,  with $t<\hat\sigma_{i}$, the length of $\hat p_i(t)$ may in the next step:
\begin{enumerate}
\item[1)] Stop growing: if $X_{t+1} \in V(\tree^{W}_{\hat\sigma_{i-1}})$, equivalently $\hat \sigma_i = t+1$; 
\item[2)] Increase by one and continue: if $X_{t+1} \notin V(\tree^{W}_{\hat\sigma_{i-1}})$ and $X_{t+1} \notin V(\hat p_i(t))$; 
\item[3)] Decrease to $\ell \in \{1, \dots, |\hat p_{i}(t)|\}$ and continue: if $X_{t+1} \in V(\hat p_i(t))$ and it corresponds to the vertex in the $\ell$-th position of the current loop erased path. Note that if $X_{t+1} = X_{t}$, the walker stepped through a self-loop and thus $|\hat p_{i}(t+1)| = |\hat p_{i}(t)|$ (the path length stayed constant). 
\end{enumerate}

On a complete graph with $n$ vertices and self-loops, for $i > 0$ and every $k \in \{i, \ldots, n-1\}$  we have that $\PP\big( \hat \sigma_i =t+1\mid  |\tree^{W}_{\hat \sigma_{i-1}}|=k \big) = \frac{k}{n}$. Moreover, if we define $\Delta \hat p_i (t):=|\hat p_i(t+1)| - |\hat p_i(t)|$, then for every $h \in \{1, \ldots, n-k-1\}$: 
\begin{align*}
&\PP\big(\Delta \hat p_i (t) = 1 \big| |\hat p_i(t)|=h, |\tree^{W}_{\hat \sigma_{i-1}}|=k \big) = \frac{n-k-h}{n}\;,  
\\
&\PP\big(\Delta \hat p_i (t) = -\ell  \big| |\hat p_i(t)|=h, |\tree^{W}_{\hat \sigma_{i-1}}|=k \big) = \frac{1}{n}\;, && \forall \ell \in \{0, 1, \ldots, h-1\}\;.
\end{align*}


Given the above transition rules for the length of the loop erasure path, the branch length distribution in Wilson's algorithm reduces to the analysis of the absorbing Markov chain illustrated in Figure~\ref{fig:mc_wilson} where the states correspond to $|\hat p_i(t)|$, namely the length of the $i$-th branch at time $t$. The probability $\PP(|\hat p_{i}(\hat \sigma_i)| = h \big | |\tree^{W}_{\hat \sigma_{i-1}}| = k)$ that the $i$-th branch has length $h$, given that the number of vertices in the previously constructed tree is $k$, corresponds to the probability of making a transition to the $stop$ state from state $h$ (i.e., probability of \emph{absorption} from state $h$).
\begin{figure}
\centering
\begin{tikzpicture}[auto,node distance=20mm, thick,main node/.style={circle,draw,minimum size=.6cm,inner sep=0pt]}]
    \node[main node] (1) {$1$};
    \node[main node] (2) [right of=1] {$2$};
    \node[main node] (3) [right of=2] {$3$};
    \node[main node] (4) [right of=3] {$4$};
    \node[minimum height = .6cm, minimum width = 1.3cm] (5) [right of=4] {$\cdots$};
    \node[ellipse,draw, inner sep = 0pt, minimum height = .6cm] (6) [right of=5] {$n-k$};
    
    \node[ellipse,draw, inner sep = 0pt, minimum height = .6cm] (7) [below right = 6em and 1.5em of 3] {$stop$};
    
    \path
    
    (1) edge[->,loop,out=195,in=178,color=red,min distance=5mm,looseness=15] node{} (1)
    (2) edge[->,loop,out=195,in=178,color=red,min distance=5mm,looseness=15] node{} (2)
    (3) edge[->,loop,out=195,in=178,color=red,min distance=5mm,looseness=15] node{} (3)
    (4) edge[->,loop,out=195,in=178,color=red,min distance=5mm,looseness=15] node{} (4)
    (6) edge[->,loop,out=190,in=178,color=red,min distance=5mm,looseness=15] node{} (6)
    
    (1) edge[->,bend left] node{$\frac{n-k-1}{n}$} (2)
    (2) edge[->,bend left] node{$\frac{n-k-2}{n}$} (3)
    (3) edge[->,bend left] node{$\frac{n-k-3}{n}$} (4)
    (4) edge[->,bend left,in=160] node[anchor=center, midway, above]{$\frac{n-k-4}{n}$} (5)
    (5) edge[->,bend left] node[anchor=center, midway, above]{$\frac{1}{n}$} (6)
    
    (4) edge[->,bend left,color=red] node{} (3)
    (4) edge[->,bend left,out=45,color=red] node{} (2)
    (4) edge[->,bend left,out=60,color=red] node{} (1)
    
    (3) edge[->,bend left,color=red] node{} (2)
    (3) edge[->,bend left,out=45,color=red] node{} (1)
    
    (2) edge[->,bend left,color=red] node{} (1)
    
    (6) edge[->,bend left,color=red,out=15] node{} (4)
    (6) edge[->,bend left,color=red,out=20] node{} (3)
    (6) edge[->,bend left,color=red,out=25] node{} (2)
    (6) edge[->,bend left,color=red,out=30] node{} (1)
    
    (1) edge[->,color=blue,out=270,in=165] node{} (7)
    (2) edge[->,color=blue,out=270,in=148] node{} (7)
    (3) edge[->,color=blue,out=270,in=120] node{} (7)
    (4) edge[->,color=blue,out=270,in=60] node{} (7)
    (6) edge[->,color=blue,out=270,in=15] node{} (7);
\end{tikzpicture}
\caption{Absorbing Markov chain representing Wilson's algorithm construction of a branch $\hat p_i$ in a complete graph with $n$ vertices, given that $|\tree^{W}_{\hat\sigma_{i-1}}| = k$. The numbered states represent the branch length $h$, and $stop$ is an absorbing state indicating the end of the  branch construction. The process always starts from $h=1$. The blue, black and red arrows correspond to the events $1), 2)$ and $3)$,
respectively.}
\label{fig:mc_wilson}
\end{figure}
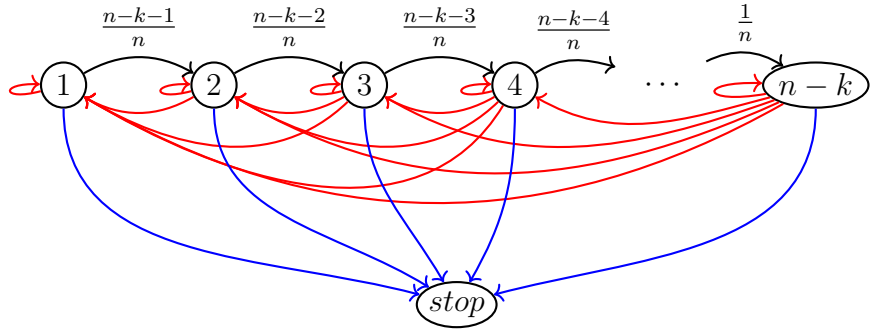 

Consider the ergodic Markov chain in Figure~\ref{fig:modified_mc_wilson} obtained by \textit{lumping} together the $stop$ state and the $h=1$ state. The probability of absorption from state $h$ in the original chain (starting from state $h=1$) corresponds to the probability that, in the ergodic chain (starting from $h=1$), the blue edge incident to state $h$ is the first blue edge traversed. This latter probability in the ergodic chain matches exactly its stationary distribution. Intuitively, since the ergodic chain regenerates every time a blue edge is traversed, the probability that a specific blue edge is the first one traversed corresponds to the long term relative frequency of times that it is traversed when compared to the other blue edges. This relative frequency, in turn, is proportional to the long term relative frequency of times that the chain visits state $h$. Since the probability of crossing a blue edge given that the chain is in a specific state does not depend on the state (the transition probability associated to every blue edge is the same), the proportionality constant is the same for every state. Thus, the  long term relative frequency of times that the blue edge incident to state $h$ is traversed is the same as the stationary distribution of $h$ in the ergodic chain.
\begin{figure}[H]
\centering
\begin{tikzpicture}[auto,node distance=20mm, thick,main node/.style={circle,draw,minimum size=.6cm,inner sep=0pt]}]
    \node[main node] (1) {$1$};
    \node[main node] (2) [right of=1] {$2$};
    \node[main node] (3) [right of=2] {$3$};
    \node[main node] (4) [right of=3] {$4$};
    \node[minimum height = .6cm, minimum width = 1.3cm] (5) [right of=4] {$\cdots$};
    \node[ellipse,draw, inner sep = 0pt, minimum height = .6cm] (6) [right of=5] {$n-k$};
    
    \path
    
    (1) edge[->,loop,out=195,in=178,color=red,min distance=5mm,looseness=15] node{} (1)
    (2) edge[->,loop,out=195,in=178,color=red,min distance=5mm,looseness=15] node{} (2)
    (3) edge[->,loop,out=195,in=178,color=red,min distance=5mm,looseness=15] node{} (3)
    (4) edge[->,loop,out=195,in=178,color=red,min distance=5mm,looseness=15] node{} (4)
    (6) edge[->,loop,out=190,in=178,color=red,min distance=5mm,looseness=15] node{} (6)
    
    (1) edge[->,bend left] node{$\frac{n-k-1}{n}$} (2)
    (2) edge[->,bend left] node{$\frac{n-k-2}{n}$} (3)
    (3) edge[->,bend left] node{$\frac{n-k-3}{n}$} (4)
    (4) edge[->,bend left,in=160] node[anchor=center, midway, above]{$\frac{n-k-4}{n}$} (5)
    (5) edge[->,bend left] node[anchor=center, midway, above]{$\frac{1}{n}$} (6)
    
    (6) edge[->,bend left,color=red,out=15,looseness=.9] node{} (4)
    (6) edge[->,bend left,color=red,out=20,looseness=.9] node{} (3)
    (6) edge[->,bend left,color=red,out=25,looseness=.9] node{} (2)
    (6) edge[->,bend left,color=red,out=30,looseness=.9] node{} (1)
    
    (4) edge[->,bend left,color=red,looseness=.9] node{} (3)
    (4) edge[->,bend left,out=45,color=red,looseness=.9] node{} (2)
    (4) edge[->,bend left,out=60,color=red,looseness=.9] node{} (1)
    
    (3) edge[->,bend left,color=red,looseness=.9] node{} (2)
    (3) edge[->,bend left,out=45,color=red,looseness=.9] node{} (1)
    
    (2) edge[->,bend left,color=red,looseness=.9] node{} (1)
    
    (1) edge[->,loop left,in=211,out=228,looseness=15,color=blue,min distance=5mm] node{} (1)
    (2) edge[->,bend left,in=130,out=45,color=blue] node{} (1)
    (3) edge[->,bend left,in=130,out=60,color=blue] node{} (1)
    (4) edge[->,bend left,in=130,out=75,color=blue] node{} (1)
    (6) edge[->,bend left,in=130,out=35,color=blue] node{} (1);
\end{tikzpicture}
\caption{Ergodic Markov chain obtained by lumping together state $stop$ and state $1$ in the absorbing chain in Figure \ref{fig:mc_wilson}.} 
\label{fig:modified_mc_wilson}
\end{figure}
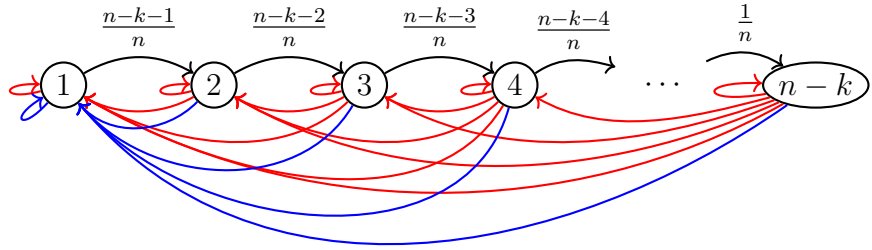

The above procedure establishes a clear strategy to determine the branch length distribution of \emph{Wilson} on complete graphs, which consists in determining the stationary distribution of the Markov chains illustrated in Figure~\ref{fig:modified_mc_wilson}. Note that the Markov chain depends only on $n$ and the number of vertices already in the tree, namely $k$. Nevertheless, explicitly solving the balance equations and finding a closed-form solution for the stationary distribution is often not trivial. However, verifying that a candidate distribution is a stationary distribution for a Markov chain is rather easy. The next lemma shows that the distribution in Lemma~\ref{lem:urn-branch}, i.e., the length of the path attached in the tree in the urn process, is the stationary distribution for the corresponding Markov chain.

\begin{lemma}\label{lem:branch_wilson}
For every $k \in \{1, \ldots, n-1\}$ the stationary distribution of the corresponding Markov chain depicted in Figure~\ref{fig:modified_mc_wilson} is:
\begin{align*}
\pi(h) = \frac{(k+h)(n-k-1)!}{n^{h}(n-k-h)!}\;, && \text{ for $h \in \{1, \ldots, n-k\}$}\;.     
\end{align*}
\end{lemma}
\begin{proof}
Every state $h$ with $h\neq 1$, has one incoming \textit{black arrow} from the state $h-1$ and one \textit{red arrow} from each state $h+ \ell$. The state $h=1$, in contrast, has $n-k$ incoming \textit{blue arrows} and $n-k-1$ \textit{red arrows}. Therefore,  the balance equations are:
\begin{equation*}
    \pi(h) = \pi(h-1)\frac{n-k-(h-1)}{n-1}+\sum_{\ell=h+1}^{n-k}\pi(\ell)\frac{1}{n-1}\;, \qquad h\neq 1\;,
\end{equation*}

\begin{equation*}
    \pi(h) = \sum_{\ell=2}^{n-k}\pi(\ell)\frac{k+1}{n-1} + \pi(1)\frac{k}{n-1}\;,  \qquad  h=1\;,
\end{equation*}
and the claim easily follows.
\end{proof}

It is easy to show that the transition probabilities for increasing and decreasing the branch size while it is being constructed on a complete graph without self-loops yields a Markov chain that has exactly the same balance equations as the Markov chain depicted in Figure~\ref{fig:modified_mc_wilson}. Thus, the branch size distribution without self-loops is identical to the case with self-loops (for complete graphs). 

\section{Aldous-Broder and Wilson equivalence and a hybrid algorithm}
\label{sec:equivalence}

Lemma~\ref{lem:branch_wilson} showed that branch sizes generated by \emph{Wilson} are identical to the number of balls painted with a given color in the urn process, conditioned on their respective pasts (i.e., number of vertices already in the tree, number of balls already painted). Recall from Theorem~\ref{thm:urn aldous equivalence} that the latter distribution is also identical to the conditional branch length distribution of \textit{Aldous-Broder} when considering a complete graph. Thus, the following interesting result follows immediately:
\equivalence

\begin{proof}
From Lemma~\ref{lem:branch_wilson} and Theorem~\ref{thm:urn aldous equivalence}.
\end{proof}

Note that the step-by-step construction of the UST on complete graphs by \textit{Aldous-Broder} and \textit{Wilson} are fundamentally distinct. However, Theorem~\ref{thm:equivalence} shows that if these two processes are observed at specific stopping times, the ``transient'' (partial) trees built by both algorithms are identically distributed. Interestingly, these stopping times correspond to the time a branch is constructed in either algorithm. Moreover, this equivalence regards also the labels of the vertices and not only the (partial) tree structure. Of course, both algorithms must have the same initialization, in the sense that vertex $r$ must the chosen identically by both algorithms. 

%

\subsection{Hybrid Algorithm}
\label{sec:hybrid}

Theorem~\ref{thm:equivalence} shows that \emph{Aldous-Broder} and \emph{Wilson} are statistically equivalent on the appropriate stopping times when running on complete graphs. This suggests that a UST can be constructed by one algorithm until a certain stopping time and then finished by the other algorithm. This gives rise to a hybrid algorithm that can switch between \emph{Aldous-Broder} and \emph{Wilson} (and back) at the corresponding stopping times. 

In particular, consider the following hybrid algorithm that starts with \emph{Aldous-Broder}, constructs $i$ branches, and then switches to \emph{Wilson} to finish generating the random tree:

\vspace{5pt}\hspace{-.8cm}\textbf{Hybrid}$\bm{(\G = (\V,\E), i)}$\textbf{:}
\begin{itemize}
    \item[0)] Set ${\E_X} \leftarrow \varnothing$ and $X_0 \leftarrow r$, with $r \in \V$ chosen arbitrarily;
    \item Until $t=\sigma^{\rm in}_i$:
        \begin{itemize}
            \item Run a simple random walk $(X_{t})_{t>0}$ on $\G$, starting at $X_{0}$ and for each edge $e = \{X_{t},X_{t+1}\}$ such that $X_{t+1} \neq X_{k}$ for all $k \leq t$,  set ${\E_X} \leftarrow {\E_X} \cup \{e\}$;
        \end{itemize}
    \item Set $\V_X \leftarrow \{X_0,\dots,X_{\sigma^{\rm in}_i}\}$;
    \item Return \textbf{Wilson($\G, \tree_o = (V_X,E_X)$)};
\end{itemize}
Note that \emph{Wilson} receives as a parameter the partial tree constructed by \emph{Aldous-Broder} and starts execution given this partial tree (see algorithm in Section~\ref{sec:introduction}). The next proposition states that the \emph{Hybrid} algorithm returns a uniform spanning tree of the complete graph $\G$. 

\begin{proposition}
    Let $\tree^{AB}_{\sigma^{\rm in}_i}$ be a random tree of the complete graph $G$ with self-loops, constructed by the first $i$ branches of \textit{Aldous-Broder} algorithm. \textit{Wilson} algorithm with initial condition $\tree^{AB}_{\sigma^{\rm in}_i}$ returns a UST($G$).
\end{proposition}
\begin{proof}
    Note that \textbf{Wilson$\left(G,\tree^{W}_{\hat{\sigma}_{i}}\right)$} $\buildrel d \over =$ \textbf{Wilson$\left(G,\tree^{AB}_{\sigma^{\rm in}_{i}}\right)$} since, as the inputs are statistically equivalent (from Theorem \ref{thm:equivalence}), it must also be the case for the outputs. Moreover, \textbf{Wilson$\left(G,\tree^{W}_{\hat{\sigma}_{i}}\right)$} corresponds exactly to \textit{Wilson's} algorithm (starting with the first $i$ branches constructed by the same algorithm). Therefore, the output of \textbf{Wilson$\left(G,\tree^{AB}_{\sigma^{\rm in}_{i}}\right)$} is also equivalent to the one of \textit{Wilson's} algorithm, a UST$(G)$.
\end{proof}

The proposed hybrid algorithm has a particularly interesting running time complexity. Note that the most time consuming phase of \textit{Aldous-Broder} and \textit{Wilson} are very similar: at the end, the random walk in \textit{Aldous-Broder} must find the last remaining vertex to be added to the tree, and at the beginning, the random walk in \textit{Wilson} must find the single anchor vertex to start the tree. On a complete graph with $n$ vertices, it is easy to conclude that both these times are geometrically distributed, with success probability $1/n$. In contrast, at the beginning the random walk of \textit{Aldous-Broder} will more quickly add vertices to tree (due to loop erasure of \textit{Wilson}), while at the end \textit{Wilson} will do so more quickly (due to random walk revisits in \textit{Aldous-Broder}). Thus, the Hybrid algorithm has the potential do reduce the overall running time since it avoids the most time consuming phases of \textit{Aldous-Broder} and \textit{Wilson} and leverages their more effective phases. In what follows, this intuition is formalized. 

The following lemma is due to Wilson~\cite{wilson1996generating} and characterizes the running time of Wilson's algorithm in terms of the number steps taken by the random walk:
\begin{lemma}[Wilson~\cite{wilson1996generating}]
\label{lem:wilson time complexity}
Let $\omega(G)$ be the \textit{mean hitting time} of $G$. The expected number of steps taken by the random walk in \textbf{Wilson($\G$, $\tree_o = \varnothing)$} is $2\omega(G)$.
\end{lemma}
On a complete graph on $n$ vertices, the \textit{mean hitting time} is $n$. Thus, the random walk takes on average $2n$ steps to return a spanning tree when running \emph{Wilson}. For \textit{Aldous-Broder} on a complete graph, it is well known that the expected number of random walk steps required to return a spanning tree is $nH_n$, where $H_n$ denotes the $n$-th harmonic number (this comes from the equivalence between the cover time of a random walk on a complete graph and the \textit{coupon collector's problem}).

The following proposition establishes the advantages of the Hybrid algorithm. Indeed, it avoids the long phases of both algorithms and exploits their good phases, creating a synergy to efficiently generate a UST. In fact, the Hybrid algorithm is more efficient than either of the algorithms alone. 
\begin{proposition}
    \label{thm: hybrid time complexity}
    The expected number of random walk steps made by \textit{Hybrid($G,1$)} is 
    $n + \Theta(\sqrt{n})$.
\end{proposition}
\begin{proof}
    The expected number of steps made by \textbf{Hybrid($G,1$)} is the sum of the average number of steps it takes for \textbf{Aldous-Broder($G$)} to construct the first branch and the average number of steps \textbf{Wilson($G, \tree_o$)} takes to complete the tree.
    
    Consider an urn with $n$ labelled balls and the process of drawing balls one at a time uniformly at random with replacement. Note that the time required for \textit{Aldous-Broder} to construct the first branch is identical to the time required to draw the first repeated ball. It is known that this has expected time $\sqrt{\pi n/2} + \Theta(1/\sqrt{n})$ (it is a variation of the birthday problem; see for example Section 1.2.11.3 in \cite{knuth1973art1}).
    
    From Lemma \ref{lem:wilson time complexity}, the expected number of steps it takes for \textit{Wilson} to build a uniform spanning tree of a complete graph $G$ is $2n$. Moreover, on average, $n$ of those steps are made to construct the first branch. Therefore, $n$ additional steps on average are required to finish the tree after the construction of the first branch of \emph{Aldous-Broder}. Hence, the expected number of random walk steps required by \textbf{Hybrid($G,1$)} to build a UST($G$) is $n + \Theta(\sqrt{n})$.
\end{proof}

In order to provide more insight into the running time of these algorithms, consider the average number of steps taken by the random walk to include at least $k$ edges in the spanning tree. Figure~\ref{fig:hybrid_simulation} shows a simulation result directly comparing the three algorithms on a complete graph with $n=1000$ vertices (results shown are the average over ten thousand tree generations for each algorithm). As expected, \textit{Aldous-Broder} is efficient in the beginning when almost every step brings a new edge to the tree, but requires a long time in the end, to finish the tree. In contrast, \textit{Wilson} is linear on $n$ but it requires a long time include the first edges in the tree (exactly 1000 steps, on average). The Hybrid algorithm starts efficiently, following along \textit{Aldous-Broder}, and remains efficient , switching to \textit{Wilson} after the first branch is created. Note that \textit{Hybrid} is approximately two times faster than \textit{Wilson}, taking roughly $1000$ steps to build the tree.

\begin{figure}
\centering
\includegraphics[width=.6\textwidth]{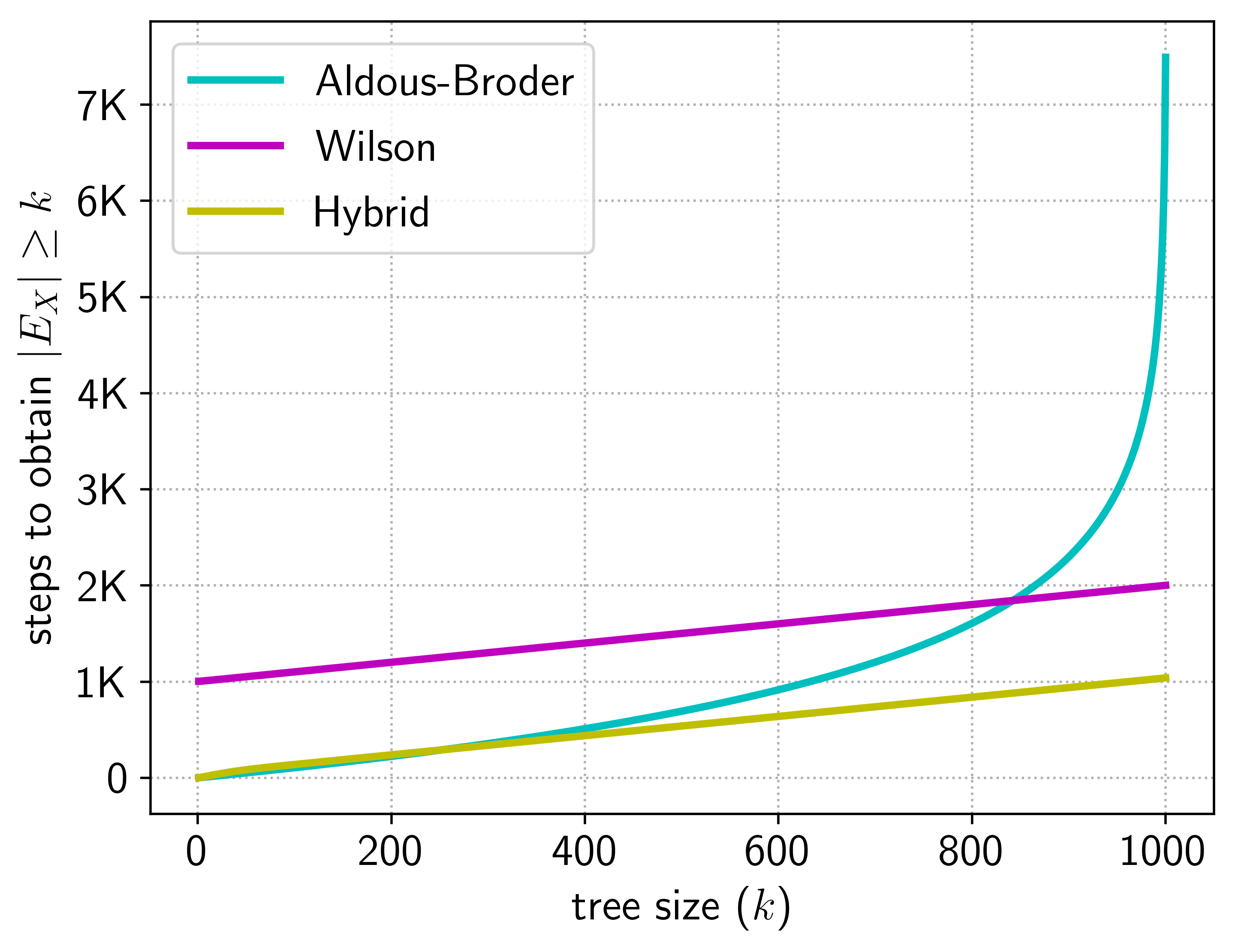}
\caption{Average running time measured in number of random walk steps to generate a UST. Each curve shows the average number of random walk steps required by each algorithm to include at least $k$ edges in the tree being constructed ($n=10^3$; average over $10^4$ independent tree generations).}
\label{fig:hybrid_simulation}
\end{figure}

While the proposed hybrid algorithm first runs \textit{Aldous-Broder} and then switches to \textit{Wilson}, it is perfectly possible to have a reverse hybrid algorithm, where \textit{Wilson} runs and constructs a few branches and then \textit{Aldous-Broder} finishes the job. In this case, an uniform random vertex within the built tree would have to be chosen to start the random walk after switching to \textit{Aldous-Broder} at an appropriate stopping time. In fact, it is also possible to go back and forth between the algorithms, always switching at the appropriate stopping time. However, from the perspective of the running time, the proposed hybrid is the most efficient. 

\section{A two-stage framework for generating uniform spanning trees}
\label{sec:framework}

The previous section presented a two-stage approach to construct uniform spanning trees for complete graphs: (i) build $i$ branches of the spanning tree using \textit{Aldous-Broder}; (ii) switch to \textit{Wilson} to finish the tree. The strategy not only maintains the uniformity across the spanning trees, but also takes less time to generate the tree. Can this \textit{Hybrid} algorithm be used to construct uniform spanning trees for any graph? 

While \textit{Hybrid} clearly constructs random spanning trees for arbitrary graphs (and is intuitively more efficient than either \textit{Aldous-Broder} or \textit{Wilson}), there is no guarantee that such trees will be uniformly distributed. Unfortunately, the random trees generated by the first branch of Aldous-Broder and Wilson algorithms are not identically distributed in general. In particular, Theorem~\ref{thm:equivalence} does not hold for arbitrary graphs as the following simple example shows. 

Let $G_0$ be the graph depicted below:
\begin{figure}[H]
\centering
\begin{tikzpicture}[scale=1.4, shorten >=0.5pt,  >=stealth,  semithick
]

\tikzstyle{every state}=[scale=0.1,draw, fill]



\node[state] at (0, 0)   (1)  {};
\node at (0,-0.3)   {$1$};
\node[state] at (0.7, 0)   (2)  {};
\node at (0.7,-0.3)   {$2$};

\node at (0.3,0.4)   {$G_0$};

\node[state] at (1.2, 0.5)   (3)  {};
\node at (1.2,0.8)   {$3$};
\node[state] at (1.2, -0.5)   (4)  {};
\node at (1.2,-0.8)   {$4$};
 \path[-] (1)  edge    node  {} (2)
 (2)  edge    node  {} (3)
 (2)  edge    node  {} (4)
 (4)  edge    node  {} (3)
 ;
\end{tikzpicture}
\end{figure}
Let  $\tree^{AB}_{\sigma_1^{\rm in}}$ be the random sub-tree generated by the first branch of  \emph{Aldous-Broder} algorithm on $G_0$,  starting from a uniformly chosen vertex, and  $\tree^{W}_{\hat{\sigma}_1}$  the random sub-tree generated by the first branch of \emph{Wilson}'s algorithm on $G_0$, then 
\begin{figure}[H]
\centering
\begin{tikzpicture}[scale=1.1, shorten >=0.5pt,  >=stealth,  semithick
]
\tikzstyle{every state}=[scale=0.1,draw, fill]



\node at (-1, 0) {$\PP\Bigg( \tree^{AB}_{\sigma_1^{\rm in}}=\quad$};
\node[state] at (-0.1, -0.1)   (1)  {};
\node at (-0.1,-0.35)   {$1$};
\node[state] at (0.5, -0.1)   (2)  {};
\node at (0.5,-0.35)   {$2$};

\node[state] at (0.9, 0.35)   (3)  {};
\node at (0.9,0.05)   {$3$};
\node at (1.5, 0) {$\quad\Bigg) = \frac{1}{12}$};
 \path[-] (1)  edge    node  {} (2)
 (2)  edge    node  {} (3)
 ;

 
\node at (3.1,0) {while}; 

\begin{scope}[xshift = 6cm]
    \node at (-1, 0) {$\PP\Bigg( \tree^{W}_{\hat{\sigma}_1}=\quad$};
    \node[state] at (-0.1, -0.1)   (1)  {};
    \node at (-0.1,-0.35)   {$1$};
    \node[state] at (0.5, -0.1)   (2)  {};
    \node at (0.5,-0.35)   {$2$};
    
    \node[state] at (0.9, 0.35)   (3)  {};
    \node at (0.9,0.05)   {$3$};
    \node at (1.5, 0) {$\quad\Bigg) = \frac{1}{9}\;.$};
    \path[-] (1)  edge    node  {} (2)
    (2)  edge    node  {} (3)
    ;
\end{scope}
\end{tikzpicture}
\end{figure}
Note that if we generate the first branch using \textit{Wilson}, i.e., $\tree^{W}_{\hat{\sigma}_1}$, and then run \textbf{Wilson($G, \tree^{W}_{\hat{\sigma}_1}$)} the final tree will be an uniform spanning tree of $G$, as expected. However, if we generate the first branch of \textit{Aldous-Broder}, i.e., $\tree^{AB}_{\sigma
^{\rm in}_1}$, and then run \textbf{Wilson($G, \tree^{AB}_{\sigma^{\rm}_1}$)}, the final spanning tree is not necessarily uniform when an arbitrary graph is considered, as the above example shows. Even so, can this idea of a two-stage procedure be saved? In what follows, this question is addressed and a promising answer is provided.  

Given a connected graph $G$, let $\Ssub_\G$ denote the set of all possible sub-trees (not necessarily spanning) of $G$. We shall denote by $\subtree$ an element of $\Ssub_\G$, while $\spatree$ shall denote an element of $\T_\G$, i.e., a spanning tree of $\G$. With a slight abuse of notation, we shall write $\subtree \subseteq \spatree$ to denote that $\subtree$ is a sub-tree of $\spatree$. Given  a  sub-tree $\subtree \in \Ssub_\G$, we denote by $\T_{\G}(\subtree)=\{\spatree \in \T_G: \spatree \supseteq  \subtree\}$, the set of all spanning trees of $G$ which contain the sub-tree $\subtree$, and by $|\T_{\G}(\subtree)|$ its cardinality. 

Let $Y$ be a random variable taking values in $\Ssub_G$, whose distribution satisfies
the following condition: for all $\spatree,\spatree' \in \T_\G$
\begin{align}\label{eq:stage_one}
    \sum_{\substack{\subtree  \in \Ssub_\G\\ \T_\G(\subtree) \ni \spatree}}\frac{1}{|\T_\G(\subtree)|}\PP(Y=\subtree)= \sum_{\substack{\subtree  \in \Ssub_\G\\ \T_\G(\subtree) \ni \spatree'}}\frac{1}{|\T_\G(\subtree)|}\PP(Y=\subtree)\;, 
\end{align}
Intuitively, the above condition says that $Y$ cannot differentiate between $\spatree$ and $\spatree'$ when generating sub-trees. The condition is subtle because different sub-trees $\subtree$ will induce different sets of spanning trees $\T_\G(\subtree)$ that have different sizes. Note that it is not sufficient for every tree to contain $Y$ with the same probability. The condition also imposes a restriction on the sizes $|\T_\G(\subtree)|$ of the induced tree sets. Thus, $Y$ must not introduce bias with respect to $\spatree$ on the weighted average across all sub-trees of $\spatree$. Despite being apparently very restricted, we will show below that the condition is met, for example, by the sub-trees created by Wilson's algorithm in the stopping times that define the branches.

The following proposition provides a two-stage procedure to generate a uniform spanning tree of an arbitrary graph $G$.
\begin{proposition}\label{prop:twostage-general}
The following two-stage ($\twostage$) procedure returns a uniform spanning tree of a connected graph $\G$. 
\begin{itemize}
    \item[1.] Draw a random variable $Y$ taking values in $\Ssub_\G$ whose distribution satisfies Equation~\eqref{eq:stage_one};
    \item[2.] Draw a spanning tree of $G$ from the set $\T_\G (Y)$ uniformly.
\end{itemize}
\end{proposition}

\begin{proof}
Let $\tree$ denote the random spanning tree of $\G$ returned by the two-stage procedure (note that the two-stage procedure always return an element of $\T_\G$). 
Let $\PP_{\twostage}(T=\spatree)$ denote the probability that the two-stage procedure returns $\spatree \in \T_\G$. 
To show that  $\tree$ is uniformly distributed on $\T_\G$, it suffices to show that $\PP_{\twostage}(T=\spatree)= \PP_{\twostage}(T=\spatree')$, for every $\spatree, \spatree' \in \T_\G$.  Given $\spatree \in \T_\G$, it holds that
\begin{align*}
\PP_{\twostage}(T=\spatree)&= \PP_{\twostage}\left(T=\spatree, \bigcup_{\substack{\subtree  \in \Ssub_\G\\ \T_\G(\subtree) \ni \spatree}}\{Y=\subtree\} \right)
\\
&= \sum_{\substack{\subtree  \in \Ssub_\G\\ \T_\G(\subtree) \ni \spatree}}\PP_{2.}(T=\spatree|Y=\subtree)\PP_{ 1.}(Y=\subtree)=  \sum_{\substack{\subtree  \in \Ssub_\G\\ \T_\G(\subtree) \ni \spatree}}\frac{1}{|\T_\G(\subtree)|}\PP_{1.}(Y=\subtree) \;,
\end{align*}
and the claim  follows form Equation~\eqref{eq:stage_one}.\hfill
\end{proof}

While Proposition~\ref{prop:twostage-general} provides a two-stage procedure to generate uniform spanning tree of any graph $\G$, it is not clear whether this procedure can give rise to an efficient algorithm. In particular, both the first and second stages must be implemented efficiently.

\medskip 
\underline{\sc Second-Stage:}
Interestingly, \emph{Wilson} is an efficient algorithm that can be used to implement the second stage in Proposition~\ref{prop:twostage-general}. In particular, given a connected graph $\G$,  and $\subtree\in \Ssub_\G$ a sub-tree (not spanning) of ${\G}$, \textbf{Wilson($G, \subtree$)} returns a spanning tree of $\G$ in the set $\{\spatree \in \T_{\G}: \spatree  \supseteq   \subtree\}$ uniformly. This is formally stated in the next lemma. 

\secondstep

\begin{proof}
The proof is based on the following two steps: 
\begin{enumerate}
\item[a)] We collapse the sub-tree $\subtree$ into a single vertex, denoted by $r$, ignoring its internal edges while keeping all the edges connecting vertices in $\V(\subtree)$ with vertices in  $\V\setminus \V(\subtree)$. Note that this procedure gives rise to a new graph $\G'$ which may have multiple edges incident on $r$. In particular, there exists a set of (external) vertices $\widehat \V \subset \V \setminus \V(\subtree)$ such that the number of edges between $v \in \widehat \V$ and $r$ is the same as the number of neighbors that $v$ has in $V(\subtree)$, denoted by $d'_v$. Thus, $v \in \widehat \V$ has $d'_v$ multiple edges incident on $r$ in $\G'$.  

From the perspective of a simple random walk, the multiple edges incident to a vertex $v \in  \widehat \V$ could be treated as a single weighted edge, with weight equal to $d'_v$; this weighted graph is denoted by ${\G}'_*$. Note that all other edges have weight 1.
\item[b)] We use a variation of Wilson's algorithm for weighted graph that generates a spanning tree of the weighted graph with probability proportional to the product of the edge weights in the corresponding tree~\cite{wilson1996generating}.   
\end{enumerate}

Note that every tree in the set $\{\spatree \in \T_{\G}: \spatree \supset \subtree\}$ corresponds to one and only one spanning tree of $\G'$, and vice versa (assuming each of the multiple edges
can be identified individually). Thus, it is sufficient to show that \textit{Wilson} on $\G'$, with initial condition $r$, returns a spanning tree in $\T_{\G'}$ uniformly.

The set of spanning trees of $\G'$ can be partitioned accordingly to which of the vertices in $\widehat \V$ are connected to $r$, ignoring which specific edge among the multiple edges are in the spanning tree. Specifically,   for $\spatree \in \T_{\G'} $,  if we denote by   $N_r(\spatree)$  the set of  neighbors of $r$ in $\spatree$, we have that 
\[
\T_{\G'} = \bigcup_{I\subseteq \widehat \V} \underset{:= \,S_I}{\underbrace{\left\{ \spatree \in \T_{\G'} : N_r(\spatree)=I \right\}} }
%
\]
Note that $\{S_I: I \subseteq \widehat \V \}$ (the set of spanning trees of $\G'$ in which $r$ is connected to the vertex $I$) is a partition of $\T_{\G'}$, and thus for any spanning tree $\spatree \in \T_{\G'}$  there exists a unique $I \subseteq \widehat \V$ such that $\spatree \in S_I$; specifically $\spatree \in  S_{N_r(\spatree)}$. Hence, if $T$ denotes the random spanning tree generated by Wilson's algorithm on $\G'$, we have that 
$ \PP(\tree=\spatree)= \PP\left(\tree=\spatree | S_{N_r(\spatree)}\right) \PP\left(S_{N_r(\spatree)}\right)$, for any $\spatree \in \T_{\G'}$. 

Now let us observe that $\PP\left(S_{N_r(\spatree)}\right)$ corresponds to the probability that Wilson's algorithm on the weighted graph $\G'_*$ (no more multiple edges) returns a specific  spanning tree in  $\T_{\G'_*}$. By point $b)$,   the probability that  Wilson's algorithm on the weighted graph $\G'_*$ returns $\tilde \spatree \in \T_{\G'_*}$ is equal to  $C\prod_{v\in N_r(\tilde \spatree)}d'_v$, where $C$ is a normalizing constant; in fact, the edges $\{r,v\}$ of the tree, with $v \in N_r(\tilde \spatree)$ have weights $d'_v$, whereas all other edges have unitary weights. 
As far as $\PP\left(\tree=\spatree | S_{N_r(\spatree)}\right)$ is concerned, since the random walk is simple, for any $v \in N_r(\spatree)$,  the walker  chooses one (and only one) of the possible multiple edges incident on $r$ uniformly, and independently across the different   $v \in N_r(\spatree)$. Thus, $\PP\left(\tree=\spatree| S_{N_r(\spatree)}\right)= 1/\left(\prod_{v\in N_r(\spatree)}d'_v\right)
$.  
Hence, for any $\spatree \in \T_{\G'}$ the probability  $\PP(\tree=\spatree)$ does not depend on $\spatree$, and thus is uniform on $\T_{\G'}$. \hfill
\end{proof}

\medskip  
\underline{\sc First-Stage:} The main concern with the first stage is constructing and sampling a random variable whose distribution satisfies Equation~\eqref{eq:stage_one}. Below we list some examples.

\medskip 
{\bf Example 1 (Wilson's branches):} Let $\G$ be an arbitrary connected graph $\G$ and let $\tree^{W}_{\hat{\sigma}_i}$ be the random sub-tree of $\G$ generated by the first $i$ branches of \emph{Wilson}'s algorithm. Then  $\tree^{W}_{\hat{\sigma}_i}$ satisfies Equation~\eqref{eq:stage_one}. To see the latter, note that if we run the two-stage procedure, when the first stage consists in drawing $\tree^{W}_{\hat{\sigma}_i}$ and the second stage in running Wilson's algorithm with initial condition  $\tree^{W}_{\hat{\sigma}_i}$, the resulting algorithm is equivalent to  Wilson's (classical) algorithm on $\G$  starting from a single vertex, whose final outcome is indeed uniform on $\T_\G$. Thus,
if we denote by $\tree$ the random tree generated by the two-stage procedure when the first stage is Wilson's first $i$ branches and the second is Wilson's algorithm with the obtained initial condition, we have that  for every $\spatree, \spatree' \in \T_\G$, it holds that $\PP_{\twostage }(\tree=\spatree)=\PP_{\twostage}(\tree=\spatree')$. Moreover 
\begin{align*}
\PP_{\twostage}(T=\spatree)= \sum_{\substack{\subtree  \in \Ssub_\G\\ \T_\G(\subtree) \ni \spatree}}\PP_{2.}(T=\spatree|\tree^{W}_{\hat{\sigma}_i}=\subtree)\PP(\tree^{W}_{\hat{\sigma}_i}=\subtree)=  \sum_{\substack{\subtree  \in \Ssub_\G\\ \T_\G(\subtree) \ni \spatree}}\frac{1}{|\T_\G(\subtree)|}\PP(\tree^{W}_{\hat{\sigma}_i}=\subtree) \;,
\end{align*}
and thus, $\tree^{W}_{\hat{\sigma}_i}$ satisfies Equation~\eqref{eq:stage_one}.

\medskip 
{\bf Example 2 (Aldous-Broder  branches on complete graphs):} Let $\G$ be a complete graph  an  let $\tree^{AB}_{\sigma^{\rm in}_i}$ be the random sub-tree of $\G$ generated by the first $i$ branches of \emph{Aldous-Broder}  algorithm. Then  $\tree^{AB}_{\sigma^{\rm in }_i}$ satisfies Equation~\eqref{eq:stage_one}, for every $i$. The latter follows directly from Example~1, together with Theorem~\ref{thm:equivalence}. 

It is worthwhile mentioning that a similar result will not hold in general graphs. Specifically, if  $\tree^{AB}_{\sigma^{\rm in}_i}$ denotes the random sub-tree  generated by the first $i$ branches of \emph{Aldous-Broder}  algorithm on an arbitrary graph  $\G$, the distribution of  $\tree^{AB}_{\sigma^{\rm in}_i}$ will not, in general,  satisfy Equation~\eqref{eq:stage_one}.  For example,  the sub-tree generated by the first branch of \emph{Aldous-Broder}   algorithm on $G_0$, given at the beginning of this section, (starting from a uniform vertex), does not satisfy Equation~\eqref{eq:stage_one}.

\medskip 
{\bf Example 3 (Uniform edge in edge-transitive  graphs):} Let $\G=(V,E)$ be an edge-transitive graph, an let $Y$ be a random variable uniformly distributed on $E$. Then  $Y$ satisfies Equation~\eqref{eq:stage_one}. As a matter of fact, for every $e \in E$, we have that $\PP(Y=e)=1/|E|$ and moreover, since the graph $\G$ is edge transitive, for every two spanning trees  $\spatree,\spatree' \in \T_\G$ it holds that 
\[
\sum_{\substack{e  \in E\\ \T_\G(e) \ni \spatree}}\frac{1}{|\T_\G(e)|}= \sum_{\substack{e  \in E\\ \T_\G(e) \ni \spatree'}}\frac{1}{|\T_\G(e)|} = \frac{n-1}{|\T_\G(e_o)|}\;,
\]
where $e_o$ is an arbitrary edge of $E$. The result follows since every spanning tree has $(n-1)$ edges, and thus appears in exactly $(n-1)$ terms in the sum. Moreover, since the graph is edge transitive, the number of spanning trees containing an specific edge $e \in E$ is the same considering any other edge $e' \in E$ due to the automorphism between $e $ and $e'$. 

Interestingly, choosing an uniform edge for a graph that is not edge-transitive will not necessarily make the requirements, as the next simple example shows.
Consider the following graph $\G$ and its spanning trees $\spatree_1, \dots,  \spatree_8$:

\begin{figure}[H]
\centering
\begin{tikzpicture}[auto,node distance=1cm,
        thick,main node/.style={circle,fill=blue!20,draw,minimum size=.5cm,inner sep=0pt],scale=.7},scale=.7,baseline=-2cm]

    \node[main node] (1) {$1$};
    \node[main node] (2) [right of=1] {$2$};
    \node[main node] (3) [below of=1] {$3$};
    \node[main node] (4) [right of=3] {$4$};
    \node[above=.3em,font=\large] at (current bounding box.north) {$\G$};

    \path[-]
    (1) edge node {} (2)
        edge node {} (3)
    (2) edge node {} (3)
    (4) edge node {} (2)
        edge node {} (3);
\end{tikzpicture}
\hspace{1cm}
\begin{tikzpicture}[auto,node distance=1cm, thick,main node/.style={circle,fill=yellow!20,draw,minimum size=.5cm,inner sep=0pt,scale=.7]},scale=.7]
    \matrix[column sep=.6cm,row sep=0.45cm]{
    \node[main node] (11) {$1$};
    \node[main node] (12) [right of=11] {$2$};
    \node[main node] (13) [below of=11] {$3$};
    \node[main node] (14) [right of=13] {$4$};
    \node[above=.3em,font=\large] at (current bounding box.north) {$\spatree_1$};

    \path[-]
    (11) edge node {} (12)
        edge node {} (13)
    (12)    edge node {} (14); &
    
    \node[main node] (21) {$1$};
    \node[main node] (22) [right of=21] {$2$};
    \node[main node] (23) [below of=21] {$3$};
    \node[main node] (24) [right of=23] {$4$};
    \node[above=.3em,font=\large] at (current bounding box.north) {$\spatree_2$};

    \path[-]
    (21) edge node {} (22)
        edge node {} (23)
    (23) edge node {} (24); &
    
    \node[main node] (31) {$1$};
    \node[main node] (32) [right of=31] {$2$};
    \node[main node] (33) [below of=31] {$3$};
    \node[main node] (34) [right of=33] {$4$};
    \node[above=.3em,font=\large] at (current bounding box.north) {$\spatree_3$};

    \path[-]
    (32) edge node {} (31)
        edge node {} (33)
        edge node {} (34); &
    
    \node[main node] (41) {$1$};
    \node[main node] (42) [right of=41] {$2$};
    \node[main node] (43) [below of=41] {$3$};
    \node[main node] (44) [right of=43] {$4$};
    \node[above=.3em,font=\large] at (current bounding box.north) {$\spatree_4$};

    \path[-]
    (41) edge node {} (42)
    (42) edge node {} (43)
    (43) edge node {} (44); \\
    
    \node[main node] (51) {$1$};
    \node[main node] (52) [right of=51] {$2$};
    \node[main node] (53) [below of=51] {$3$};
    \node[main node] (54) [right of=53] {$4$};
    \node[above=.3em,font=\large] at (current bounding box.north) {$\spatree_5$};

    \path[-]
    (51) edge node {} (52)
    (52) edge node {} (54)
    (53) edge node {} (54); &
    
    \node[main node] (61) {$1$};
    \node[main node] (62) [right of=61] {$2$};
    \node[main node] (63) [below of=61] {$3$};
    \node[main node] (64) [right of=63] {$4$};
    \node[above=.3em,font=\large] at (current bounding box.north) {$\spatree_6$};

    \path[-]
    (61) edge node {} (63)
    (62) edge node {} (63)
    (64) edge node {} (62); &
    
    \node[main node] (71) {$1$};
    \node[main node] (72) [right of=71] {$2$};
    \node[main node] (73) [below of=71] {$3$};
    \node[main node] (74) [right of=73] {$4$};
    \node[above=.3em,font=\large] at (current bounding box.north) {$\spatree_7$};

    \path[-]
    (73) edge node {} (71)
    (73) edge node {} (72)
    (74) edge node {} (73); &
    
    \node[main node] (81) {$1$};
    \node[main node] (82) [right of=81] {$2$};
    \node[main node] (83) [below of=81] {$3$};
    \node[main node] (84) [right of=83] {$4$};
    \node[above=.3em,font=\large] at (current bounding box.north) {$\spatree_8$};

    \path[-]
    (81) edge node {} (83)
    (83) edge node {} (84)
    (84) edge node {} (82);\\
    };
\end{tikzpicture}
\end{figure}
The cardinality of $\T_\G(e)$ for $e \in \{(1,2), (1,3), (2,4), (3,4) \}$ is five since each of these edges appear in exactly five spanning trees. However, the cardinality of $\T_\G((2,3)))$ is four as edge $(2,3)$ appears in exactly four spanning trees. Thus, choosing an edge uniformly at random and then choosing a spanning tree uniformly at random given the chosen edge will generate a bias towards spanning that have the edge $(2,3)$. In fact, the probability of choosing $\spatree_{i}$ is 12/100 when $i \in \{1,2,5,8\}$ and 13/100 when $i  \in \{3,4,6,7\}$.


\medskip 
{\bf Example 4 (Aldous-Broder first branch on a cycle):} %
Let $\G$ be a cycle graph on $n$ vertices, henceforth denoted by $\mathcal{C}_n$ and $(X_t)_{t\geq 0}$ be a simple random walk on $\mathcal{C}_n$.  Let  $\sigma^{\rm in}_1:=\inf\{t\geq 0: X_t \in \cup_{k=0}^{t-1}\{X_k\}\}$ and denote by $\tree^{AB}_{\sigma^{\rm in}_1}$ the sub-tree generated  by the first-entrance edges up to time $\sigma^{\rm in}_1$. The random variable $\tree^{AB}_{\sigma^{\rm in}_1}$ satisfies Equation~\eqref{eq:stage_one}. To see the latter, let us first observe that if 
$\subtree \in \Ssub_{\mathcal{C}_n}$ is a  sub-tree of $\mathcal{C}_n$, and  $\E(\subtree)$ denotes  the set of edges of $\subtree$, it holds  that 
\[
\PP\left(\tree^{AB}_{\sigma^{\rm in}_1}=\subtree\right)= \frac{1}{n} 2^{-|\E(\subtree)|}\;. 
\]
Moreover, given $\subtree \in \Ssub_{\mathcal{C}_n}$, the set of spanning tree of $\mathcal{C}_n$ containing $\subtree$ satisfies $|\T_{\mathcal{C}_n}(\subtree)|= n - |E(\subtree)|$. Using that all spanning trees of $\mathcal{C}_n$ are isomorphic to each other, for every $\spatree, \spatree' \in \T_{\mathcal{C}_n}$,  we obtain that
\[
\sum_{\substack{\subtree  \in \Ssub_{\mathcal{C}_n}\\ \T_{\mathcal{C}_n}(\subtree) \ni \spatree}}\frac{1}{n-|\E(\subtree)|} \frac{1}{n} 2^{-|\E(\subtree)|}=
\sum_{\substack{\subtree  \in \Ssub_{\mathcal{C}_n}\\ \T_{\mathcal{C}_n}(\subtree) \ni \spatree'}}\frac{1}{n-|\E(\subtree)|} \frac{1}{n} 2^{-|\E(\subtree)|}
\;.
\]

{\bf Example 5 (Sub-tree generated by a simple random walk on a cycle):} %
Let $\G$ be a cycle graph on $n$ vertices, henceforth denoted by $\mathcal{C}_n$.
Given $k\geq 0$, let us consider the random sub-tree $\tree_k^X$  induced by a simple random walk $(X_t)_{t\geq 0}$ on $\mathcal{C}_n$ up to time $k$,  starting from a uniformly chosen vertex of $\mathcal{C}_n$, i.e.,  $\tree_k^X$ is the tree with vertex set $\cup_{t=0}^k \{X_t\}$ and edge set  $\cup_{t=1}^{k} \{X_{t-1}, X_t\}$. Then, if $\eta$ denotes the cover time of $\mathcal{C}_n$,  for every $k\geq 1$ the random variable $T^X_{(k\wedge \eta)-1}$ satisfies Equation~\eqref{eq:stage_one}. To see the latter, let $\subtree \in \Ssub_{\mathcal{C}_n}$ be a  sub-tree of $\mathcal{C}_n$, and denote by $\E(\subtree)$ the set of edges of $\subtree$. Note that, 
$\PP\left(T^X_{(k\wedge \eta)-1}=\subtree\right)=0$, if $k<  |E(\subtree)|$, whereas $\PP\left(T^X_{(k\wedge \eta)-1}=\subtree\right)$ is  a function only of $n, k$ and $|E(\subtree)|$, if $k\geq \E(\subtree)$ (i.e., given $n$ and $k$ sub-trees having the same number of edges will have the same probability). Moreover, given $\subtree \in \Ssub_{\mathcal{C}_n}$, the set of spanning tree of $\mathcal{C}_n$ containing $\subtree$ satisfies $|\T_{\mathcal{C}_n}(\subtree)|= n - |E(\subtree)|$. Using that all spanning trees of $\mathcal{C}_n$ are isomorphic to each other, for every $\spatree, \spatree' \in \T_{\mathcal{C}_n}$,  we obtain that, for any fixed $k$, 
\[
\sum_{\substack{\subtree  \in \Ssub_{\mathcal{C}_n}\\ \T_{\mathcal{C}_n}(\subtree) \ni \spatree}}\frac{1}{n-|\E(\subtree)|} \PP\left(T^X_{(k\wedge \eta)-1}=\subtree\right)=
\sum_{\substack{\subtree  \in \Ssub_{\mathcal{C}_n}\\ \T_{\mathcal{C}_n}(\subtree) \ni \spatree'}}\frac{1}{n-|\E(\subtree)|} \PP\left(T^X_{(k\wedge \eta)-1}=\subtree\right)
\;.
\]

Examples 4 and 5 consider a cycle graph $G$ with $n$ vertices. Uniform spanning trees of $G$ can be very efficiently generated by simply removing from $G$ an edge chosen uniformly at random. Thus, while the two-stage procedure in these two examples does not lead to an efficient algorithm, the goal here is to provide a random variable based on random walks that satisfies Equation~\eqref{eq:stage_one} along with an efficient sampling procedure for this variable. This could shed light on crafting such random variable for a broader class of graphs. 

   




\subsection{Speed up on transitive graphs}
\label{sec:hypercubes}
This section provides a specific application of the two-stage framework, using the result shown in Example 3. In particular, the two-stage framework can generate USTs of transitive graphs (edge transitive) more efficiently than Wilson's algorithm, illustrating its potential.

Recall from Proposition \ref{prop:twostage-general} that the two-stage framework consists of the following two steps: (1) drawing a sub-tree $Y$ of $G$ according to a distribution which satisfies Equation~\ref{eq:stage_one};  (2) drawing a spanning tree uniformly from $\T_G(Y)$. 
Consider the following instance of the two-stage framework applied to edge-transitive graphs: (1) draw an edge $Y$ of the graph uniformly at random; (2) run \textit{Wilson} on the graph with initial condition $Y$. We refer to this algorithm as \textbf{Edge-Wilson}$\bm{(G)}$, where $G$ denotes the edge-transitive graph. Since in edge-transitive graphs choosing an edge uniformly at random satisfies the first stage condition (see, Example 3), and along with Lemma~\ref{lem:second_step}, we have that \textbf{Edge-Wilson}$\bm{(G)}$ returns a UST of $G$ (the edge-transitive assumption is important to assure that \textbf{Edge-Wilson}$\bm{(G)}$ returns a UST of $G$).

To construct the first branch, \textbf{Wilson}$\bm{(G, \varnothing)}$ (i.e., {\em Wilson's} algorithm with a single vertex as initial condition) takes in average a time equal to the mean hitting time of $G$, denoted by $\omega(G)= \sum_{j\in V(G)}\sum_{i\in V(G)}\pi(i)\pi(j) \mathbb{E}_i(h_{j})$, where $\pi$ is the stationary distribution of a symmetric random walk on $G$ and $h_j$ is the hitting time of $j$ (it is well known that $\omega(G)=  \sum_{j\in V(G)}\pi(j) \mathbb{E}_i(h_{j})$ since the rhs does not depend on $i$ \cite{aldous-fill-2014}). 
On the other hand, the time taken by {\em Edge-Wilson} to construct the first branch is equivalent to the time \textbf{Wilson}$\bm{(G,e)}$ (i.e.,  Wilson's algorithm where the initial tree is random edge) takes to construct the first branch, which is $\phi(G):= \sum_{i\in V(G)}\sum_{\{u,v\}\in E(G)}\pi(i)\nu(\{u,v\})\mathbb{E}_i (h_{\{u,v\}})$, where $\pi$ is the stationary distribution of a symmetric random walk on $G$, $\nu$ is the uniform distribution on $E(G)$, and $\mathbb{E}_i (h_{\{u,v\}})$ is the expected hitting time of the edge $\{u,v\}$ when the walker starts from $i$. 

Intuitively, $\omega(G)\geq \phi(G)$ since it is easier for the random walk to hit an edge than a vertex. But how much is gained when constructing the first branch using a random initial edge as opposed to random initial vertex? 
The following lemma gives a closed-form expression for $\omega(G)-\phi(G)$ for an arbitrary graph $G$ and another for when $G$ is transitive.  

\begin{lemma}\label{lem:speed-up}
Let $G$ be a graph and $\{v,u\} \in E(G)$. It holds that
\[
\omega(G)  -  \phi(G) = \sum_{\{u,v\}\in E(G)}\frac{1}{|E(G)|}\mathbb{E}_v(h_{u}) \sum_{i\in V(G)}\pi(i)\mathbb{P}_i(h_{v}<h_{u})\;.
\]
Moreover, if $G$ is transitive we obtain that $\omega(G) - \phi(G)=  \frac{\mathbb{E}_v(h_{u})}{2}$.
\end{lemma}
\begin{proof}
The proof follows  from noticing that for every $i \in V(G)$ and every $\{u,v\} \in E(G)$ it holds that
\[
\mathbb{E}_i(h_{\{u,v\}})= \mathbb{E}_i(h_{u})- \mathbb{E}_v(h_{u}) \mathbb{P}_i(h_{v}<h_{u})\;.
\]  
Multiplying both sides by $\pi(i)$ (stationary distribution of a symmetric random walk on $G$) and $\nu(\{u,v\})$ (uniform distribution on $E(G)$) and summing over all $i \in V(G)$ and $\{u,v\}\in E(G)$ we obtain that
\[
\phi(G)= \omega(G) - \sum_{\{u,v\}\in E(G)}\frac{1}{|E(G)|}\mathbb{E}_v(h_{u}) \sum_{i\in V(G)}\pi(i)\mathbb{P}_i(h_{v}<h_{u})\;,
\]
where we used that \[\sum_{\{u,v\}\in E(G)} \nu(\{u,v\}) \mathbb{E}_i(h_{u})= \sum_{u\in V(G)}\sum_{v: \{u,v\}\in E(G)} \frac{1}{2}\nu(\{u,v\}) \mathbb{E}_i(h_{u})= \sum_{u\in V(G)} \mathbb{E}_i(h_{u}) \pi(u)\;,
\]
since $\nu(\{u,v\})$ is uniform over the edges of the graph. 
If $G$ is vertex-transitive we have that $\sum_{i\in V(G)}\pi(i)\mathbb{P}_i(h_{v}<h_{u})=1/2$  (see,  \cite{aldous_1989}) and also that  $\mathbb{E}_v(h_{u})$ does not depend on the specific edge (edge-transitive); thus $\phi(G)= \omega(G) - \frac{\mathbb{E}_v(h_{u})}{2}$.
\end{proof}

Lemma~\ref{lem:speed-up} allows to quantify the speed-up of \textbf{Edge-Wilson}$\bm{(G)}$, as compared to \textbf{Wilson}$\bm{(G, \varnothing)}$,  in constructing the first branch. In some specific cases  this gain can be explicitly computed. For example: 
\begin{itemize}
\item If $G=K_n$, i.e. a  complete graph on $n$ vertices, it is well-known that 
$ \omega(K_n)=n$ and that for $\{u,v\}\in E(K_n)$, $ \mathbb{E}_v(h_u)=\omega(K_n)$. Thus, from Lemma~\ref{lem:speed-up} we obtain $\frac{\phi(K_n)}{\omega(K_n)}=\frac{1}{2}$.

\item If $G=K_{\frac{n}{2}\frac{n}{2}}$, i.e., a  complete balanced bipartite graph on $n$ (even) vertices, for $\{u,v\}\in E(K_{\frac{n}{2}\frac{n}{2}})$ we have that $\mathbb{E}_v(h_u)=n-1$ and  $\omega(K_{\frac{n}{2}\frac{n}{2}})=n-\frac{1}{2}$.  Thus, from Lemma~\ref{lem:speed-up} we obtain $\frac{\phi(K_{\frac{n}{2}\frac{n}{2}})}{\omega(K_{\frac{n}{2}\frac{n}{2}})}=\frac{n}{2n-1}$.  

\item If $G=Q_d$, i.e. a  $d$-dimensional hypercube, it is well-known that 
$\lim_{d\rightarrow + \infty} \omega(Q_d)/2^d=1$ and that for $\{u,v\}\in E(Q_d)$, $\lim_{d\rightarrow + \infty} \mathbb{E}_v(h_u)/\omega(Q_d)=1$ (e.g., see \cite{levin2017markov}). Thus, from Lemma~\ref{lem:speed-up} we obtain $\displaystyle\lim_{d\rightarrow +\infty}\frac{\phi(Q_d)}{\omega(Q_d)}=\frac{1}{2}$. 
\end{itemize}
Thus, on a complete graph, complete balanced bipartite graph, or an hypercube, the expected number of random walk steps (i.e., running time) taken by \textbf{Edge-Wilson}$\bm{(G)}$ to construct the first branch is asymptotically half the expected number of random walk steps taken by \textbf{Wilson}$\bm{(G, \varnothing)}$ to construct the first branch. 

It is well known that \textit{Wilson} spends half of its running time to construct the first branch of the tree and half to finish the tree~\cite{wilson1996generating}. Thus, if the first branch produced by \textbf{Edge-Wilson}$\bm{(G)}$  and by \textbf{Wilson}$\bm{(G, \varnothing)}$ are ``statistically equivalent'' \textit{Edge-Wilson} would require half the time of \textit{Wilson} to construct the first branch (on complete graphs or hypercubes) and the same time as \textit{Wilson} to finish the tree, i.e.,  will require 25\% less time overall. This leads to following conjecture: 

\begin{conjecture}
Let $G$ be a transitive graph where \textit{Wilson} takes, asymptotically in the size of $G$,   half of the time to construct the first branch when its initial condition is an edge instead of a single vertex. 
Then, the expected time taken by \textbf{Edge-Wilson}$\bm{(G)}$ is asymptotically $25\%$ less than the expected time of \textbf{Wilson}$\bm{(G, \varnothing)}$, i.e.,  Wilson's algorithm on the graph $G$.  
\end{conjecture}


To corroborate the conjecture, 
%
Figure~\ref{fig:hc_simulations}(c) shows the average number of random walk steps obtained from simulating the \textit{Edge-Wilson} and \textit{Wilson} on hypercubes with different dimensions and $10^4$ independent runs. 
%
%
\begin{figure}[H]
    \centering
    \captionsetup[subfigure]{justification=centering}
    \begin{subfigure}[t]{0.32\textwidth}
        \centering
        \includegraphics[width=\textwidth]{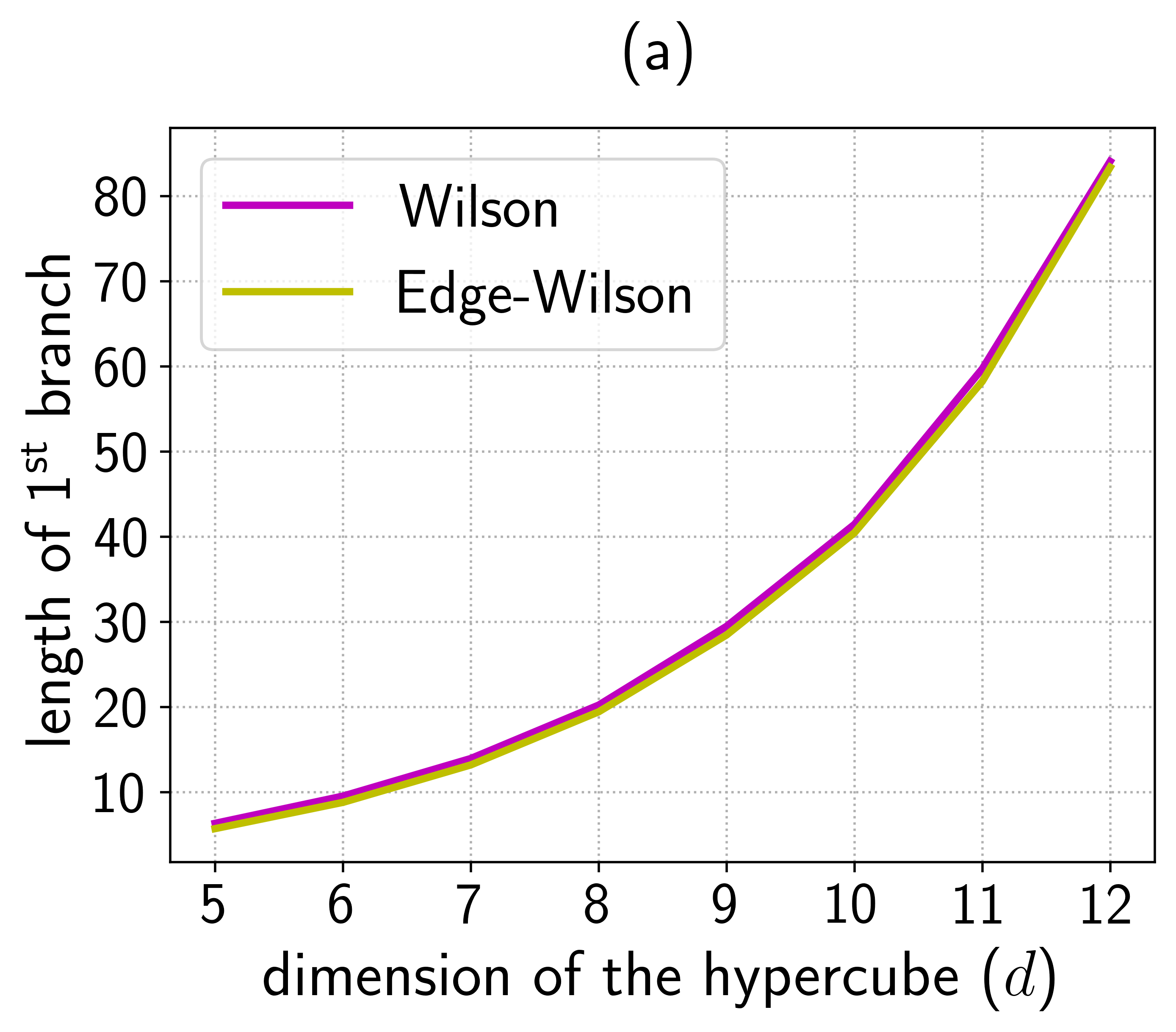}
    \end{subfigure}
    \begin{subfigure}[t]{0.32\textwidth}
        \centering
        \includegraphics[width=\textwidth]{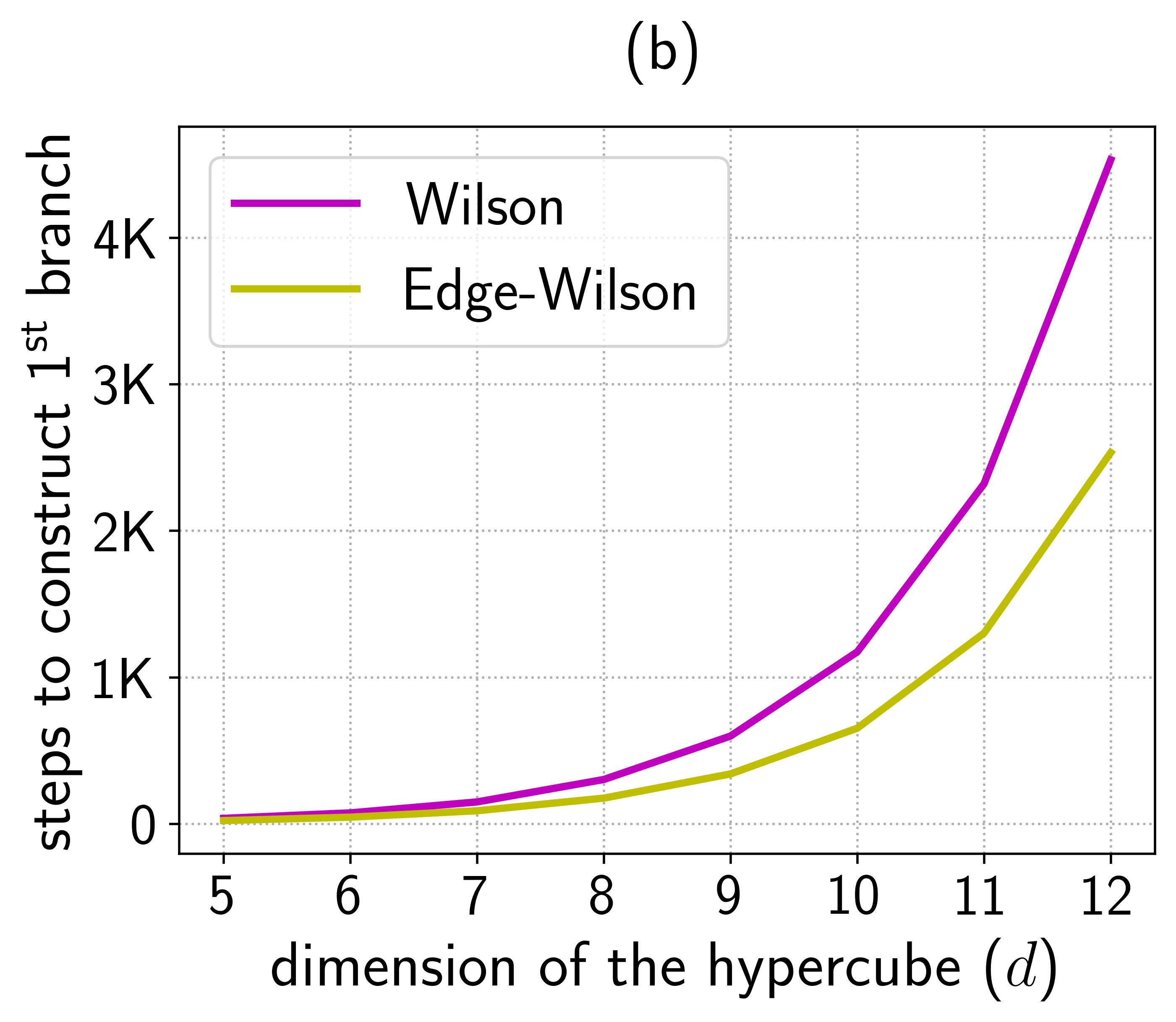}
    \end{subfigure}
    \begin{subfigure}[t]{0.32\textwidth}
        \centering
        \includegraphics[width=\textwidth]{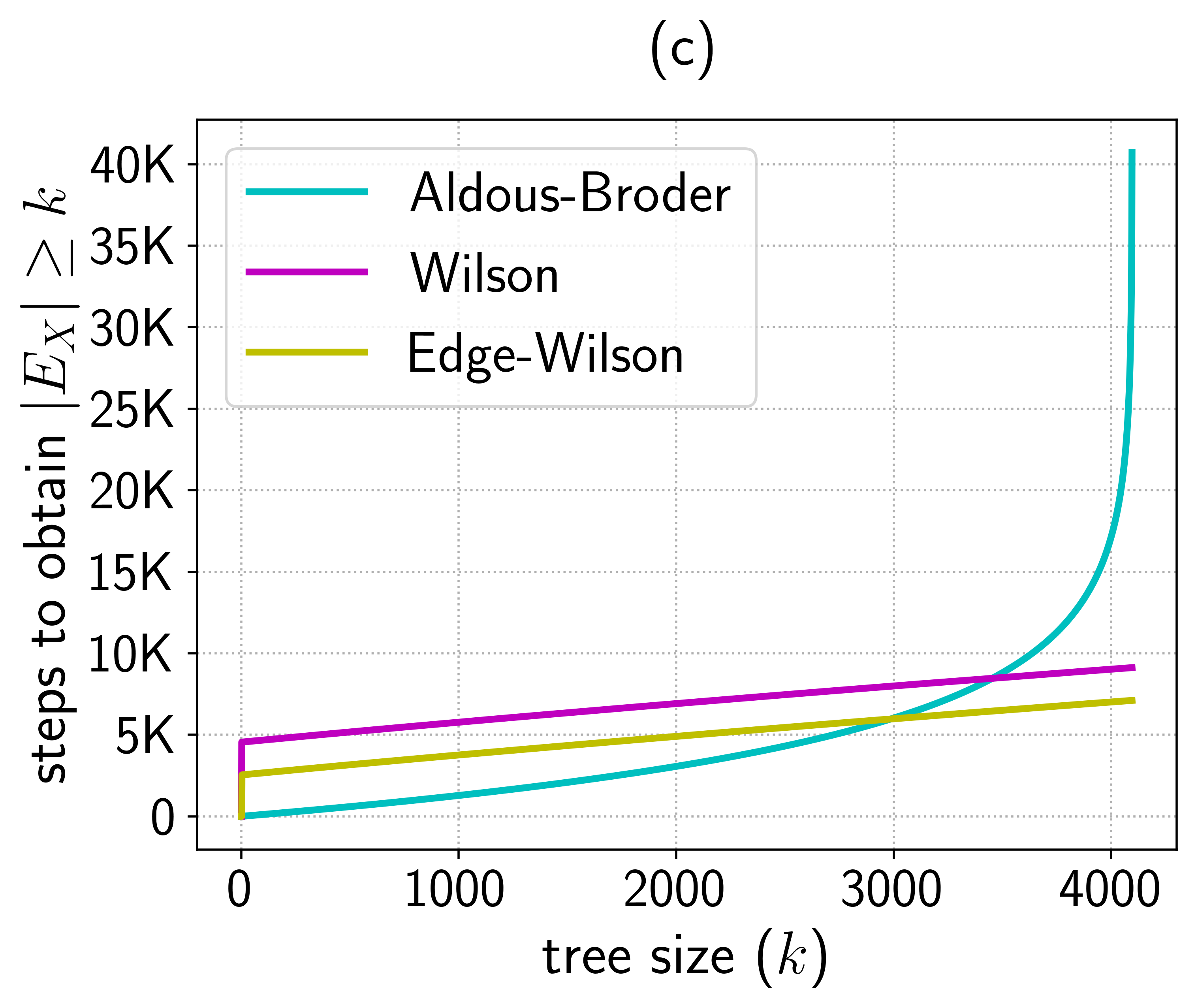}
    \end{subfigure}
    \caption{Simulation results on the hypercube: (a) average length of the first branch constructed by \textit{Edge-Wilson} and \textit{Wilson} as a function of hypercube dimension; (b) average number of steps taken by the two algorithms to construct the first branch as a function of hypercube dimension; (c) average number of steps required by each algorithm to include at least $k$ edges in the tree when running on a hypercube with dimension $d=12$. Every data point represents the average over $10^4$ independent runs.}
    \label{fig:hc_simulations}
\end{figure}
In particular for $d=12$, \textit{Edge-Wilson} required $7113.8$ steps (on average) to generate the spanning tree, whereas \textit{Wilson} required $9210.9$ steps (on average), leading to a factor (Edge-Wilson steps/Wilson steps) of approximately $0.77$. Figure~\ref{fig:hc_simulations}(b) shows the average number of steps to construct the first branch for each case. As the dimension $d$ increases, this ratio will converge to 1/2. 
Last, Figure~\ref{fig:hc_simulations}(a) provides some evidence that the first branch of both cases are statistically identical: the average length of the first branch produced by \textit{Wilson} and \textit{Edge-Wilson} are asymptotically identical, a necessary condition for them to be equivalent.

\bibliographystyle{amsplain}
\bibliography{sample.bib}

\appendix
\clearpage
\section{Random trees biased by a sub-tree structure}
\label{sec:linearly_biased}

Consider the problem of generating a random tree that has a sub-tree that is isomorphic to a given smaller tree. Can we do so in expected time that is linear on the number of vertices? For a particular and meaningful distribution over the random trees, the answer is yes! As will be discussed, this result is a direct consequence of Lemma~\ref{lem:second_step}, used in the second phase of the two-stage framework.

The \textit{Seeded-Tree} algorithm receives as input a positive integer $n$ and a tree $\subtree$ with $n'\leq n$ vertices and returns a random tree $\tree^{*}$ with $n$ vertices that has at least one sub-tree isomorphic to $\subtree$. 

\vspace{5pt}\hspace{-.8cm}\textbf{Seeded-Tree}$\bm{(n,\subtree)}$\textbf{:}
\begin{itemize}
    \item[0)] Let $\G = (\V,\E)$ be a complete graph with $n$ vertices;
    \item $\tree \leftarrow $ \textbf{Wilson($\G,\subtree$)};
    \item $\tree^{*} \leftarrow $ the random tree obtained by uniformly shuffling the labels of $\tree$;
    \item Return $\tree^{*}$.
\end{itemize}


%

\textit{Seeded-Tree} clearly generates a random tree with at least one sub-tree that is isomorphic to $\subtree$, and while the distribution over them is not uniform, the following theorem characterizes this distribution and its running time.

\begin{theorem}
    \label{thm:seeded-tree}
    $\textbf{Seeded-tree($n,\subtree$)}$ generates a random tree $\spatree$ with $n$ vertices
    with probability proportional to the number of sub-trees of $\spatree$ that are isomorphic to $\subtree$ in expected time $O(n)$.
\end{theorem}

\begin{proof}
Consider $\tree$, the random variable (random tree) denoting the output of \textbf{Wilson($G, \subtree$)}, where $\G$ is a complete graph on $n$ vertices and $\subtree$ is any sub-tree of $\G$ (not necessarily spanning). Now consider $\tree^*$, the random tree obtained by uniformly shuffling the labels of $\tree$. Clearly, since $G$ is complete, $\tree^*$ is a spanning tree in $\T_{G}$.

Consider the probability $\PP_{\subtree}\left(\tree^{*}=\spatree\right)$ that $\tree^*$ is a specific spanning tree $\spatree \in \T_{G}$. Since we shuffle the labels in the end, each sub-tree $\subtree^{*}$ of $\spatree$ that is isomorphic to $\subtree$ corresponds to one such event, then:
\begin{align*}
\PP_{\subtree}(T^{*}=\spatree) = \sum_{\substack{\subtree^{*}\in\Ssub_{\G}:\\\subtree^{*}\subseteq\;\spatree \;\wedge\; \subtree^{*}\simeq\; \subtree}}\PP\big(\text{\textbf{Wilson($G, \subtree^{*}$)}}=\spatree\big)
\end{align*}
where $\subtree^{*}\simeq\subtree$ indicates that there exists an isomorphism between $\subtree$ and $\subtree^{*}$, and $\{\text{\textbf{Wilson($G, \subtree^{*}$)}}$ $=\spatree\}$ represents the event that Wilson's algorithm with initial condition $\subtree^{*}$ returns $\spatree$.

From Lemma \ref{lem:second_step}, we know that $\PP\left(\text{\textbf{Wilson($G, \subtree^{*}$)}}=\spatree\right) = 1/|\T_{G}(\subtree^{*})|$. Moreover, the fact that $\G$ is complete tells us that $|\T_{G}(\subtree')| = |\T_{G}(\subtree'')|$ for all $\subtree'\simeq\subtree''$. Then, the probability of generating a specific tree $\spatree$ is proportional to the cardinality of the sum, which is exactly the number of sub-trees of $\spatree$ that are isomorphic to $\subtree$. Finally, note that the procedure consists of running \emph{Wilson} on a complete graph, which takes $O(n)$ expected time (as discussed in Section \ref{sec:hybrid}) and performing a random shuffle of the labels, which has $O(n)$ running time~\cite{knuth2014art2}. Thus, the expected running time of \textit{Seeded-Tree} is $O(n)$.
\end{proof}

\textit{Seeded-Tree} can be used to efficiently generate random trees with probability biased by a given sub-tree. For example, let $\subtree$ be a star graph with $k$ nodes. Running $\textbf{Seeded-tree($n,\subtree$)}$ generates spanning trees of the complete graph with probability proportional to the number of vertices with degree $k$. If $\subtree$ is a path of length $l$, $\textbf{Seeded-tree($n,\subtree$)}$ generates trees with probability proportional to the number of paths of length $l$. The generation of random trees constrained by a given structure finds application in a variety of areas, from decision making algorithms to real-life network optimization problems~\cite{li2005sampling,deo1997computation,liu2005maximizing}.











\end{document}